\documentclass[12pt,a4paper]{amsart}
\usepackage{amssymb,eucal}
\usepackage[cmtip,arrow]{xy}
\usepackage{pb-diagram,pb-xy}

\calclayout

\newcommand{\+}{\nobreakdash-}
\renewcommand{\:}{\colon}
\renewcommand{\;}{,\medspace}
\renewcommand{\.}{\text{$\mskip .5\thinmuskip$}}

\newcommand{\rarrow}{\longrightarrow}

\newcommand{\ot}{\otimes}
\newcommand{\os}{\lozenge}

\DeclareMathOperator{\Hom}{Hom}
\DeclareMathOperator{\Ext}{Ext}
\DeclareMathOperator{\Tor}{Tor}

\newcommand{\Hot}{\mathsf{Hot}}
\newcommand{\Ab}{\mathsf{Ab}}

\newcommand{\fX}{{\mathfrak X}}
\newcommand{\bfY}{{\boldsymbol{\mathfrak Y}}}

\newcommand{\boR}{{\mathbb R}}
\newcommand{\boL}{{\mathbb L}}
\newcommand{\boZ}{{\mathbb Z}}

\newcommand{\C}{{\mathcal C}}
\newcommand{\D}{{\mathcal D}}
\newcommand{\bS}{{\boldsymbol{\mathcal S}}}
\newcommand{\bT}{{\boldsymbol{\mathcal T}}}

\newcommand{\sD}{{\mathsf D}}
\newcommand{\sE}{{\mathsf E}}
\newcommand{\sJ}{{\mathsf J}}

\newcommand{\oc}{\mathbin{\text{\smaller$\square$}}}
\newcommand{\bu}{{\text{\smaller\smaller$\scriptstyle\bullet$}}}
\newcommand{\lrarrow}{\.\relbar\joinrel\relbar\joinrel\rightarrow\.}

\DeclareMathOperator{\qTor}{\mathcal T \mskip-.5\thinmuskip
  \text{\rmfamily\mdseries\fontshape{ui}\selectfont or}}
\DeclareMathOperator{\qHom}{\mathcal H \mskip-.3\thinmuskip
  \text{\rmfamily\mdseries\fontshape{ui}\selectfont om}}

\newcommand{\modl}{{\operatorname{\mathsf{--mod}}}}
\newcommand{\modr}{{\operatorname{\mathsf{mod--}}}}
\newcommand{\bimod}{{\operatorname{\mathsf{--mod--}}}}
\newcommand{\modrfp}{{\operatorname{\mathsf{mod_{fp}--}}}}

\newcommand{\fp}{{\mathsf{fp}}}
\newcommand{\fpi}{{\mathsf{fpi}}}
\newcommand{\fpp}{{\mathsf{fpp}}}
\newcommand{\inj}{{\mathsf{inj}}}
\newcommand{\proj}{{\mathsf{proj}}}
\newcommand{\fl}{{\mathsf{fl}}}
\renewcommand{\b}{{\mathsf{b}}}
\newcommand{\co}{{\mathsf{co}}}
\newcommand{\ctr}{{\mathsf{ctr}}}
\newcommand{\sico}{{\mathsf{sico}}}
\newcommand{\sictr}{{\mathsf{sictr}}}
\newcommand{\abs}{{\mathsf{abs}}}
\newcommand{\si}{{\mathsf{si}}}

\newcommand{\dinj}{{\operatorname{\mathsf{--inj}}}}
\newcommand{\dfpi}{{\operatorname{\mathsf{--fpi}}}}
\newcommand{\dproj}{{\operatorname{\mathsf{--proj}}}}
\newcommand{\dfl}{{\operatorname{\mathsf{--fl}}}}

\newcommand{\rop}{{\mathrm{op}}}
\newcommand{\sop}{{\mathsf{op}}}

\newcommand{\Section}[1]{\bigskip\section{#1}\medskip}
\theoremstyle{plain}
\newtheorem{thm}{Theorem}[section]
\newtheorem{lem}[thm]{Lemma}
\newtheorem{prop}[thm]{Proposition}
\newtheorem{cor}[thm]{Corollary}

\theoremstyle{definition}
\newtheorem{ex}[thm]{Example}

\begin{document}

\title{Coherent rings, fp-injective modules, \\ dualizing complexes, and
covariant \\ Serre--Grothendieck duality}

\author{Leonid Positselski}

\address{Department of Mathematics, Faculty of Natural Sciences,
University of Haifa, Mount Carmel, Haifa 31905, Israel; and
\newline\indent Laboratory of Algebraic Geometry, National Research
University Higher School of Economics, Moscow 117312; and
\newline\indent Sector of Algebra and Number Theory, Institute for
Information Transmission Problems, Moscow 127051, Russia}

\email{posic@mccme.ru}

\begin{abstract}
 For a left coherent ring $A$ with every left ideal having
a countable set of generators, we show that the coderived category of
left $A$\+modules is compactly generated by the bounded derived
category of finitely presented left $A$\+modules (reproducing
a particular case of a recent result of \v St'ov\'\i\v cek
with our methods).
 Furthermore, we present the definition of a dualizing complex of
fp\+injective modules over a pair of noncommutative coherent rings $A$
and $B$, and construct an equivalence between the coderived category
of $A$\+modules and the contraderived category of $B$\+modules.
 Finally, we define the notion of a relative dualizing complex of
bimodules for a pair of noncommutative ring homomorphisms $A\rarrow R$
and $B\rarrow S$, and obtain an equivalence between
the $R/A$\+semicoderived category of $R$\+modules and
the $S/B$\+semicontraderived category of $S$\+modules.
 For a homomorphism of commutative rings $A\rarrow R$, we also
construct a tensor structure on the $R/A$\+semicoderived category
of $R$\+modules.
 A vision of semi-infinite algebraic geometry is discussed in
the introduction.
\end{abstract}

\maketitle

\tableofcontents

\section*{Introduction}
\medskip

\subsection{{}}
 The philosophy of semi-infinite homological algebra, as elaborated
in the book~\cite{Psemi}, tells that semi-infinite homology and
cohomology theories are naturally assigned to mathematical objects
``of semi-infinite nature'', meaning objects that can be viewed as
extending in both a ``positive'' and a ``negative'' direction with
some ``zero position'' in between, perhaps defined up to
a finite movement.
 In application to algebraic geometry, one thinks of
a ``semi-infinite algebraic variety'' as an ind-pro-algebraic
variety or an ind-scheme of ind-infinite type, with the scheme
or pro-variety variables forming the ``negative direction'' and
the ind-variety variables belonging to the ``positive'' one.
 Thus the simplest example of a semi-infinite algebraic variety is
the affine/vector space of formal Laurent power series $k((z))$
over a ground field~$k$, and many more geometrically complicated
examples are supposed to be constructed using the field structure
of the Laurent power series.

 More specifically, experience seems to suggest that the ``positive''
variables have to be ``grouped together'' in some sense, forming
a well-defined ``positive subalgebra'' object in the ``semi-infinite'' 
algebra of functions or operators, like the subalgebra $zk[[z]]d/dz$
in the Lie algebra $k((z))d/dz$ of vector fields on the formal circle.
 In the context of algebraic geometry, this points to a morphism of
ind-schemes or ind-stacks $\pi\:\bfY\rarrow\fX$ with, approximately,
the following properties:
\begin{enumerate}
\renewcommand{\theenumi}{\Roman{enumi}}
 \item $\bfY$ is a large and complicated ind-scheme or ind-stack;
 \item $\fX$ is built up in a complicated way from affine schemes of
rather small size: something like an ind-Noetherian ind-scheme or
an ind-Noetherian ind-stack with a dualizing complex;
 \item the morphism $\bfY\rarrow\fX$ is locally well-behaved: one
would probably want it to be at least flat, or perhaps ``very
flat'' in the sense of~\cite[Section~1.7]{Pcosh};
 \item the \emph{fibers} of the morphism $\bfY\rarrow\fX$ are built up
in a simple way from large affine pieces: so they might be arbibrary
affine schemes, or quasi-compact semi-separated schemes, or perhaps
some kind of ``weakly proregular formal schemes'' in
the sense of~\cite{PSY,Pmgm}.
\end{enumerate}
 For example, the surjective linear map of topological vector spaces
$k((z))\rarrow k((z))/k[[z]]$ can be viewed naturally as a morphism
of ind-schemes satisfying the conditions~(I--IV).
 The discrete quotient space $k((z))/k[[z]]$ is the set of
$k$\+points of an ind-scheme of ind-finite type over~$k$, while
the fibers, isomorphic to $k[[z]]$, are the sets of $k$\+points
of affine schemes of infinite type.

\subsection{{}}
 In the algebraic formalism of~\cite{Psemi}, the main starting object
is a \emph{semialgebra} $\bS$, that is an associative algebraic
structure ``mixing algebra and coalgebra variables''.
 The ``positively indexed'' variables form a coalgebra~$\C$;
the semialgebra $\bS$ is an algebra object in the category of
bicomodules over~$\C$.
 The key structures in the categorical formalism are
the \emph{semiderived categories} of semimodules and semicontramodules
over~$\C$; these are mixtures of the co/contraderived categories
``in the direction of~$\C$'' and the conventional derived categories
``in the direction of $\bS$ relative to~$\C$''.

 In the geometric situation described above, the purpose of having
a morphism of ind-schemes or ind-stacks $\bfY\rarrow\fX$ is to
consider the semiderived category of quasi-coherent torsion
sheaves or contraherent cosheaves of contramodules on $\bfY$
\emph{relative to $\fX$}, which means ``the co- or contraderived
category along $\fX$ and the conventional derived category along
the fibers''.
 The ``semi-infinite algebraic geometry'' formalism would then feature
a ``geometric semimodule-semicontramodule correspondence'', i.~e.,
a triangulated equivalence between the two semiderived (or, if one
wishes, the \emph{semicoderived} and the \emph{semicontraderived})
categories.

 In addition, one expects to have a ``semi-infinite quasi-coherent
$\qTor$ functor'', or the double-sided derived functor of
\emph{semitensor product} of quasi-coherent torsion sheaves on~$\bfY$.
 This means a mixture of the \emph{cotensor
product}~\cite[Section~B.2.5]{Psing} of quasi-coherent torsion sheaves
along the ind-scheme/ind-stack $\fX$ with its dualizing complex and
the conventional tensor product of quasi-coherent sheaves along
the fibers.
 The derived semitensor product functor should provide a tensor
structure on the semiderived category of quasi-coherent torsion sheaves,
and the pull-back of the dualizing complex $\pi^*\D_\fX^\bu$ should be
the unit object of this tensor structure.
 One would also expect to have a double-sided derived functor of
\emph{semihomomorphisms} from quasi-coherent torsion sheaves to
contraherent cosheaves of contramodules on $\bfY$, transformed by
the derived semico-semicontra correspondence into the conventional
right derived quasi-coherent internal $\qHom$.

\subsection{{}}
 The aim of the present paper is to work out a couple of small
pieces in the above big picture.
 First of all, we attempt to show that the Noetherianness condition
in~(II) can be weakened to the coherence condition.
 The definition of a dualizing complex over a commutative coherent
ring, or a pair of noncommutative ones, is elaborated for this purpose.
 On a more technical level, we demonstrate the usefulness of
the notion of an fp\+injective module over a coherent ring.
 Secondly, we introduce the definition of a relative dualizing complex
and obtain an equivalence between the semicoderived and
the semicontraderived categories of modules in the simplest geometric
situation of a morphism of affine schemes $\bfY\rarrow\fX$.
 In addition, we construct the derived semitensor product functor in
this situation, defining a tensor structure on the semiderived category
of modules.

 Notice that the case a quasi-compact semi-separated scheme $\bfY$
over a point $\fX=*$ has been already considered
in~\cite[Section~4.6]{Pcosh} and the case of a Noetherian scheme
$\bfY$ over $\fX=*$, in~\cite[Theorem~5.8.1]{Pcosh}.
 The case of a weakly proregular (e.~g., Noetherian) affine
formal scheme $\bfY$ over a point $\fX=*$ is clarified in
the paper~\cite{Pmgm}.
 The situation of a (semi-separated or non-semi-separated)
Noetherian scheme $\bfY=\fX$ with a dualizing complex has been
considered in~\cite[Section~5.7 and Theorem~5.8.2]{Pcosh}.
 The case of a semi-separated Noetherian stack $\bfY=\fX$ with
a dualizing complex has been worked out in~\cite[Section~B.4]{Pcosh},
and the case of an ind-affine ind-Noetherian ind-scheme $\bfY=\fX$
with a dualizing complex, in~\cite[Section~D.2]{Pcosh}.
 (The reader can find an overview of these results in
the recent presentation~\cite{Psli2}.)
 The present paper adds an item or two to this list.

\medskip

\textbf{Acknowledgements.}
 The mathematical content of this paper was conceived in Moscow and
subsequently worked out in Haifa and in Brno.
 I am grateful to Henning Krause who told me about fp\+injective
modules during a workshop in Moscow in September~2011.
 I would like to thank Jan \v St'ov\'\i\v cek for sending me
his preprint~\cite{Sto} and Amnon Yekutieli for helpful discussions.
 The author was supported in part by RFBR grants in Moscow,
by a fellowship from the Lady Davis Foundation at the Technion,
and by the Grant agency of the Czech Republic under the grant
P201/12/G028 at Masaryk University in Brno.
 The author's research is supported by the Israel Science Foundation
grant~\#\,446/15 at the University of Haifa.

\Section{Fp-Injective and Fp-Projective Modules}

 Hereditary complete cotorsion theories in abelian and exact
categories~\cite{ET,Hov,Gil}, \cite[Section~1.1]{Bec}, starting with
the flat cotorsion theory in the category of modules over
an associative ring and the very flat cotorsion theory in the category
of modules over a commutative ring, and continuing with numerous
others, play an important role in the theory of contraherent
cosheaves~\cite{Pcosh}.

 The theory of fp\+injective and fp\+projective modules is one of
the classical examples of complete cotorsion
theories~\cite[Definition~3.3 and Theorem~3.4(2)]{Trl}, \cite{MD}.
 From our point of view, its importance in the study of modules over
coherent rings lies in the fact that the class of fp\+injective modules,
while often not differing very much homologically from the narrower
class of injective ones, is at the same time closed under infinite
direct sums, and in fact, even under filtered inductive
limits~\cite{Sten}.
 Thus the use of fp\+injective modules allows to work with many coherent
rings in the ways otherwise applicable to Noetherian rings only.

 This section contains preliminary material, and the proofs are sketchy.
 Filling in the details is left to the reader.

 Given an associative ring $A$, we denote by $A\modl$ the abelian
category of left $A$\+modules and by $\modr A$ the abelian category
of right $A$\+modules.
 A left $A$\+module $M$ is said to be \emph{finitely presented} if it
can be presented as the cokernel of a morphism of finitely generated
free left $A$\+modules.
 Clearly, the cokernel of a morphism from a finitely generated
left $A$\+module to a finitely presented one is finitely presented;
an extension of finitely presented left $A$\+modules is finitely
presented.

\begin{lem}  \label{fin-gen-pres-kernel}
 The kernel of a surjective morphism from a finitely generated module
to a finitely presented one is finitely generated. \qed
\end{lem}

 A ring $A$ is called \emph{left coherent} if any finitely generated
submodule of a finitely presented left $A$\+module is finitely
presented, or equivalently, if any finitely generated left ideal
in $A$ is finitely presented as a left $A$\+module.
 Whenever $A$ is a left coherent ring, the full subcategory $A\modl_\fp$
of finitely presented left $A$\+modules is closed under the kernels,
cokernels, and extensions in $A\modl$; so $A\modl_\fp$ is an abelian
category and its embedding $A\modl_\fp\rarrow A\modl$ is
an exact functor.

\begin{lem} \label{fin-gen-proj-complex}
 Let $A$ be a left coherent ring, and let $C^\bu$ be a bounded above
complex of left $A$\+modules whose cohomology modules $H^n(C^\bu)$
are finitely presented over~$A$.
 Then there exists a bounded above complex of finitely generated free
left $A$\+modules $F^\bu$ together with a quasi-isomorphism of 
complexes of $A$\+modules $F^\bu\rarrow C^\bu$. \qed
\end{lem}

\begin{cor} \label{derived-fp-fully-faithful}
 For any left coherent ring $A$, the triangulated functors between
the derived categories of bounded and bounded above complexes\/
$\sD^\b(A\modl_\fp)\rarrow\sD^\b(A\modl)$ and\/ $\sD^-(A\modl_\fp)\rarrow
\sD^-(A\modl)$ induced by the embedding of abelian categories
$A\modl_\fp\rarrow A\modl$ are fully faithful.
 Their essential images coincide with the full subcategories\/
$\sD^\b_\fp(A\modl)$ and\/ $\sD^-_\fp(A\modl)$ of complexes with
finitely presented cohomology modules in\/ $\sD^\b(A\modl)$ and\/
$\sD^-(A\modl)$. \qed
\end{cor}

 Let $A$ be a left coherent ring.
 A left $A$\+module $J$ is said to be \emph{fp\+injective}~\cite{Sten}
if the functor $\Hom_A({-},J)$ takes short exact sequences of finitely
presented left $A$\+modules to short exact sequences of abelian
groups, or equivalently, if $\Ext_A^1(M,J)=0$ for any finitely
presented left $A$\+module $M$, or if $\Ext_A^i(M,J)=0$ for all
finitely presented $M$ and all $i>0$.
 All injective modules are fp\+injective.
 The class of fp\+injective left modules over left coherent ring $A$
is closed under extensions, cokernels of injective morphisms, infinite
direct sums and products, and filtered inductive limits.
 So, in particular, the full subcategory $A\modl_\fpi$ of
fp\+injective left $A$\+modules inherits the exact category structure
of the abelian category $A\modl$.

 The next definition and the related assertions, including the rest of
this section and also Lemma~\ref{coderived-hom-computing}(b) below, are
never really used in the proofs of the main results of this paper.
 They are presented here for the sake of completeness of the exposition,
and in the belief that the related techniques will find their uses in
the future development of semi-infinite algebraic geometry.

 A left $A$\+module $P$ is said to be
\emph{fp\+projective}~\cite{Trl,MD} if the functor $\Hom_A(P,{-})$
takes short exact sequences of fp\+injective left $A$\+modules to
short exact sequences of abelian groups, or equivalently, if
$\Ext_A^1(P,J)=0$ for any fp\+injective left $A$\+module $J$, or
if $\Ext_A^i(P,J)=0$ for all fp\+injective $J$ and all $i>0$.
 All projective modules and all finitely presented modules are
fp\+projective.
 The class of fp\+projective left modules over a left coherent ring $A$
is closed under extensions, kernels of surjective morphisms, and
infinite direct sums.
 So the full subcategory $A\modl_\fpp$ of fp\+projective left
$A$\+modules inherits the exact category structure of the abelian
category $A\modl$.

 Moreover, the class of fp\+projective left $A$\+modules is closed
under \emph{transfinitely iterated extensions} in the following
sense (``of inductive limit'').
 A left $A$\+module $P$ is said to be a transfinitely iterated
extension of left $A$\+modules $M_\alpha$ if there exist
a well-ordering of the set of indices $\{\alpha\}$ and an increasing
filtration $F_\alpha P$ of the $A$\+module $P$ by its $A$\+submodules
such that one has $\bigcup_\alpha F_\alpha P=P$ and for every
index~$\alpha$ the quotient module $F_\alpha P/\bigcup_{\beta<\alpha}
F_\beta P$ is isomorphic to $M_\alpha$.
 The following result~\cite{ET,Trl} tells that there are ``enough''
fp\+injective and fp\+projective left $A$\+modules.

\begin{thm}  \label{fp-complete-cotorsion-theory}
\textup{(a)} Any left $A$\+module can be embedded into
an fp\+injective left $A$\+module in such a way that
the quotient module is fp\+projective. \par
\textup{(b)} Any left $A$\+module is the quotient module of some
fp\+projective left $A$\+module by its fp\+injective submodule. \qed
\end{thm}

 The fp\+projective modules in both parts of
Theorem~\ref{fp-complete-cotorsion-theory} are constructed as
certain transfinitely iterated extensions of finitely
presented modules.
 Hence it follows from (the proof of) part~(b) that a left
$A$\+module $P$ is fp\+projective if and only if it is 
a direct summand of a transfinitely iterated extension of
finitely presented left $A$\+modules.

\begin{lem}
 Let $A$ be a left coherent ring.
 Then any finitely generated submodule of an fp\+projective left
$A$\+module is finitely presented.
\end{lem}

\begin{proof}
 It suffices to show that any finitely generated submodule $N\subset P$
of a transfinitely iterated extension $(P,F)$ of finitely presented
left $A$\+modules $M_\alpha$ is finitely presented.
 Let $\alpha_0$~be the minimal index~$\alpha$ such that $N$ is contained
in $F_\alpha P$ (since $N$ is finitely generated, such indices~$\alpha$
exist).
 The quotient module $N/N\cap\bigcup_{\beta<\alpha}F_\beta P$ is a finitely
generated submodule of a finitely presented left $A$\+module $M_\alpha$,
and consequently, also a finitely presented $A$\+module.
 By Lemma~\ref{fin-gen-pres-kernel}, the $A$\+module
$N\cap\bigcup_{\beta<\alpha}F_\beta P$ is finitely generated; and
the assumption of induction in the ordinal $\{\alpha\}$ tells that
it is finitely presented.
 Now the $A$\+module $N$ is finitely presented as an extension of
two finitely presented $A$\+modules.
\end{proof}

\begin{lem} \label{d-conventional-hom-computing}
 Let $A$ be a left coherent ring, $P^\bu$ be a complex of fp\+projective
left $A$\+modules, and $J^\bu$ be a complex of fp\+injective left
$A$\+modules.
 Then whenever either the complex $P^\bu$ is bounded above, or
the complex $J^\bu$ is bounded below, the Hom complex\/
$\Hom_A(P^\bu,J^\bu)$ computes the groups\/
$\Hom_{\sD(A\modl)}(P^\bu,J^\bu[*])$.
\end{lem}

\begin{proof}
 One notices that the complex $\Hom_A(P^\bu,J^\bu)$ is acyclic
whenever either the complex $P^\bu$ is a bounded above complex of
projective $A$\+modules and the complex $J^\bu$ is acyclic, or
the complex $P^\bu$ is acyclic and the complex $J^\bu$ is a bounded
below complex of injective $A$\+modules.
 Therefore, the complex $\Hom_A(P^\bu,J^\bu)$ computes the groups
$\Hom_{\sD(A\modl)}(P^\bu,J^\bu[*])$ whenever either $P^\bu$ is
a bounded above complex of projective $A$\+modules, or $J^\bu$ is
a bounded below complex of injective $A$\+modules.

 Furthermore, the complex $\Hom_A(P^\bu,J^\bu)$ is acyclic whenever
either the complex $P^\bu$ is an acyclic bounded above complex of
fp\+projective left $A$\+modules and $J^\bu$ is a complex of
fp\+injective left $A$\+modules, or $P^\bu$ is a complex of
fp\+projective left $A$\+modules and $J^\bu$ is a bounded below
acyclic complex of fp\+injective left $A$\+modules.
 Since any bounded above complex of $A$\+modules is the target of
a quasi-isomorphism from a bounded above complex of projective
$A$\+modules, and any bounded below complex of $A$\+modules is
the source of a quasi-isomorphism into a bounded below complex
of injective $A$\+modules, the desired assertions follow.
\end{proof}

\Section{Coderived Category of Modules over a Coherent Ring}

 This section is our take on~\cite[Conjecture~5.9]{Kr2}.
 Notice that this conjecture of Krause's is already resolved (proven
in the coherent and disproven in the noncoherent case) by
\v St'ov\'\i\v cek in~\cite[Theorem~6.12, Corollary~6.13, and
Example~6.15]{Sto}.
 The more elementary approach below is based on the techniques of
working with derived categories of the second kind developed
in~\cite{Psing} and formulated in the form convenient for us here
in~\cite[Appendix~A]{Pcosh}, instead of the set-theoretic methods
of~\cite{Sto}.

 Given an additive category $\sE$, we denote by $\Hot(\sE)$
the homotopy category of (unbounded complexes over)~$\sE$.
 We refer to~\cite[Section~A.1]{Pcosh} for the definitions of
the \emph{coderived category} $\sD^\co(\sE)$ and
the \emph{contraderived category} $\sD^\ctr(\sE)$ of
an exact category $\sE$ with exact functors of infinite direct sum or
infinite product, respectively.
 A slightly different definition of such categories was suggested by
Becker in~\cite[Proposition~1.3.6]{Bec}; it is also used in~\cite{Sto}.
 The definitions in~\cite{Bec} have the advantage of working well
for the category of modules over an arbitrary ring (and also
CDG\+modules over an arbitrary CDG\+ring).
 Our definitions have the advantage of being more explicit and
applicable to abelian/exact categories of quite general nature.

\begin{prop} \label{coderived-right-resolutions}
 Let\/ $\sE$ be an exact category with exact functors of infinite
direct sum, and let\/ $\sJ\subset\sE$ be a full subcategory closed
under infinite direct sums.
 Assume that the full subcategory\/ $\sJ$ is closed under extensions
in\/ $\sE$, and endow it with the induced exact category structure.
 Assume further that\/ $\sJ$ is closed under the passages to
the cokernels of admissible monomorphisms in\/ $\sE$, and that any
object of\/ $\sE$ is the source of an admissible monomorphism into
an object of\/~$\sJ$.
 Then the triangulated functor\/ $\sD^\co(\sJ)\rarrow\sD^\co(\sE)$
induced by the embedding of exact categories\/ $\sJ\rarrow\sE$ is
an equivalence of triangulated categories. 
\end{prop}

\begin{proof}
 This is the assertion dual to~\cite[Proposition~A.3.1(b)]{Pcosh}.
\end{proof}

\begin{thm} \label{coderived-fp-injective}
 Let $A$ be a left coherent ring.
 Then the triangulated functor $\sD^\co(A\modl_\fpi)\rarrow
\sD^\co(A\modl)$ induced by the embedding $A\modl_\fpi\rarrow A\modl$
of the exact category of fp\+injective $A$\+modules into the abelian
category of arbitrary $A$\+modules is an equivalence of
triangulated categories.
\end{thm}

\begin{proof}
 This is a particular case of
Proposition~\ref{coderived-right-resolutions}.
 Similarly one can prove that the coderived category $\sD^\co(A\modl)$
of left CDG\+modules over a CDG\+ring $(A,d,h)$ with a left graded
coherent underlying graded ring $A$ is equivalent to the coderived
category $\sD^\co(A\modl_\fpi)$ of left CDG\+modules with fp\+injective
underlying graded modules.
 (Cf.\ the similar assertion about the contraderived categories of
flat and arbitrary CDG\+modules over CDG\+rings with coherent
underlying graded rings in~\cite[Remark~1.5]{Psing}.)
\end{proof}

 According to~\cite[Th\'eor\`eme~7.10]{GJ}, the projective dimension
of a flat module over an associative ring of the cardinality $\aleph_n$
cannot exceed~$n+1$.
 The following result is simpler, though sounds somewhat similar.

\begin{prop}  \label{cardinality-generators-ideals}
 Let $A$ be a left coherent ring such that any left ideal in $A$
admits a set of generators of the cardinality not
exceeding\/~$\aleph_n$, where $n$~is an integer.
 Then the injective dimension of any fp\+injective left $A$\+module
is not greater than\/~$n+1$.
\end{prop}

\begin{proof}
 By Baer's criterion, a left $A$\+module $K$ is injective whenever
$\Ext_A^1(A/I,K)=0$ for all left ideals $I\subset A$.
 Hence it suffices to prove that $\Ext_A^{n+2}(A/I,J)=0$ for all
left ideals $I$ and all fp\+injective left $A$\+modules~$J$.
 Any left ideal $I\subset A$ is the inductive limit of the filtered
inductive system of its finitely generated subideals $I_\alpha\subset I
\subset A$, and the quotient module $A/I$ is a filtered inductive
limit of the quotient modules $A/I_\alpha$.
 Furthermore, for any filtered inductive system of left modules
$L_\alpha$ and a left module $M$ over an associative ring $A$ there is
a spectral sequence
$$
 E_2^{p,q}=\varprojlim\nolimits_\alpha^p\Ext_A^q(L_\alpha,M)
 \Longrightarrow\Ext_A^{p+q}(\varinjlim\nolimits_\alpha L_\alpha\;M),
$$
as one can see by replacing $M$ with its right injective resolution
and computing the derived functors of inductive and projective limits
$\varinjlim_\alpha^p L_\alpha$ and
$\varprojlim_\alpha^p\Ext_A^q(L_\alpha,M)$ in terms of
the bar-constructions.
 In particular, if $\Ext_A^q(L_\alpha,M)=0$ for all~$\alpha$ and all
$q>0$, then $\Ext_A^p(\varinjlim_\alpha L_\alpha\;M)\simeq
\varprojlim_\alpha^p\Hom_A(L_\alpha,M)$.
 It remains to recall that the homological dimension of the derived
functor of projective limit along a filtered poset of
the cardinality~$\aleph_n$ does not exceed~$n+1$ \cite{Mit}.
\end{proof}

 The following result is to be compared to the discussion of
contraderived categories over coherent CDG\+rings
in~\cite[Section~3.8]{Pkoszul}.

\begin{thm} \label{coderived-homotopy-injectives}
 Let $A$ be a left coherent ring such that any fp\+injective left
$A$\+module has finite injective dimension.
 Then the triangulated functor $\Hot(A\modl_\inj)\rarrow
\sD^\co(A\modl)$ induced by the embedding of additive/exact
categories $A\modl_\inj\rarrow A\modl$ is an equivalence of
triangulated categories.
\end{thm}

\begin{proof}
 Moreover, for any CDG\+ring $(A,d,h)$ whose underlying graded ring
$A$ is left coherent and has the property that the injective
dimensions of fp\+injective graded left modules over it are finite,
the homotopy category of left CDG\+modules over $(A,d,h)$ with
injective underlying graded left $A$\+modules is equivalent to
the coderived category of CDG\+modules.
 Indeed, according to~\cite[Section~3.7]{Pkoszul} it suffices that
countable direct sums of injective (graded) $A$\+modules be of
finite injective dimensions, so it remains to recall that direct
sums of injective modules are fp\+injective.
\end{proof}

 Notice that the coderived category of $A$\+modules in the sense of
Becker~\cite[Proposition~1.3.6(2)]{Bec} is \emph{defined} as
the homotopy category of complexes of injective $A$\+modules (or
the coderived category of CDG\+modules over $A$ is defined as
the homotopy category of CDG\+modules with injective underlying
graded $A$\+modules, in the case of a CDG\+ring $A=(A,d,h)$).
 Hence our Theorem~\ref{coderived-homotopy-injectives}, when its
homological dimension condition is satisfied, makes the results of
\v St'ov\'\i\v cek~\cite[Section~6]{Sto} about Becker's coderived
category of complexes of modules over a coherent ring (or complexes
of objects of a locally coherent Grothendieck category) applicable
to the coderived category in our sense.
 In particular, our compact generation result in
Corollary~\ref{compact-generators}(b) below becomes a particular case
of~\cite[Corollary~6.13]{Sto}.

 Furthermore, it is instructive to compare the result of our
Theorem~\ref{coderived-fp-injective} with that
of~\cite[Theorem~6.12]{Sto}.
 According to Theorem~\ref{coderived-fp-injective}, our coderived
category of the abelian category of left $A$\+modules is equivalent to
our coderived category of the exact category of fp\+injective left
$A$\+modules.
 According to~\cite[Theorem~6.12]{Sto}, Becker's coderived category of
the abelian category of left $A$\+modules is equivalent to
the \emph{conventional} derived category of the exact category of
fp\+injective left $A$\+modules.
 In both cases, it is only assumed that the ring $A$ is left coherent. 

\begin{lem} \label{coderived-hom-computing}
 Let $A$ be a left coherent ring.  Then \par
\textup{(a)} for any bounded complex $P^\bu$ of finitely presented
left $A$\+modules and any complex $J^\bu$ of fp\+injective left
$A$\+modules, the Hom complex\/ $\Hom_A(P^\bu,J^\bu)$ computes
the groups\/ $\Hom_{\sD^\co(A\modl)}(P^\bu,J^\bu[*])$; \par
\textup{(b)} assuming that fp\+injective left $A$\+modules have finite
injective dimensions, for any complex $P^\bu$ of fp\+projective left
$A$\+modules and any complex $J^\bu$ of fp\+injective left $A$\+modules
the complex\/ $\Hom_A(P^\bu,J^\bu)$ computes the groups\/
$\Hom_{\sD^\co(A\modl)}(P^\bu,J^\bu[*])$.
\end{lem}

\begin{proof}
 According to (the proof of) Theorem~\ref{coderived-fp-injective},
any complex of left $A$\+modules admits a morphism with a coacyclic
cone into a complex of fp\+injective left $A$\+modules, and any
complex of fp\+injective left $A$\+modules that is coacyclic with
respect to the abelian category of arbitrary left $A$\+modules is
also coacyclic with respect to the exact category of fp\+injective
left $A$\+modules.
 Hence in both parts~(a) and~(b) it suffices to prove that
the complex $\Hom_A(P^\bu,J^\bu)$ is acyclic when (the complex $P^\bu$
satisfies the respective condition and) $J^\bu$ is a coacyclic
complex of fp\+injective left $A$\+modules.
 Furthermore, in the assumption of~(b) the exact category of
fp\+injective left $A$\+modules has finite homological dimension,
so any coacyclic (and even any acyclic) complex in it is
absolutely acyclic~\cite[Remark~2.1]{Psemi}.
 It remains to notice that the complex $\Hom_A$ from a complex of
fp\+projective left $A$\+modules to the total complex of a short
exact sequence of complexes of fp\+injective left $A$\+modules is
acyclic, and the functor $\Hom_A$ from a bounded complex of finitely
generated left $A$\+modules takes infinite direct sums of complexes
of left $A$\+modules to infinite direct sums of complexes of 
abelian groups.
\end{proof}

\begin{cor} \label{compact-generators}
 Let $A$ be a left coherent ring.  Then \par
\textup{(a)} the full subcategory of bounded complexes of finitely
presented left $A$\+modules\/ $\sD^\b(A\modl_\fp)\subset\sD^\co(A\modl)$
in the coderived category of left $A$\+modules consists of compact
objects in\/ $\sD^\co(A\modl)$; \par
\textup{(b)} assuming that fp\+injective left $A$\+modules have finite
injective dimensions, the coderived category\/ $\sD^\co(A\modl)$
is compactly generated, and its full subcategory\/ $\sD^\b(A\modl_\fp)$
is precisely its full subcategory of compact objects.
\end{cor}

\begin{proof}
 By~\cite[Lemma~A.1.2]{Pcosh}, the full subcategory of bounded below
complexes in $\sD^\co(A\modl)$ is equivalent to the category
$\sD^+(A\modl)$, in which $\sD^\b(A\modl_\fp)\subset\sD^\b(A\modl)
\subset\sD^+(A\modl)$ is a full subcategory.
 Hence the full subcategory of bounded complexes of finitely
presented left $A$\+modules in $\sD^\co(A\modl)$ is indeed equivalent
to $\sD^\b(A\modl_\fp)$.
 Since the class of fp\+injective modules is closed under infinite
direct sums in $A\modl$, the rest of part~(a) follows from
Lemma~\ref{coderived-hom-computing}(a).

 In part~(b), it is clear that the category $\sD^\b(A\modl_\fp)$
contains the images of its idempotent endomorphisms, since so does
the derived category $\sD(A\modl)$, where $\sD^\b(A\modl_\fp)$ is
an idempotent closed subcategory by Lemma~\ref{fin-gen-proj-complex}.
 So it remains to show that any complex $C^\bu$ of left $A$\+modules
such that $\Hom_{\sD^\co(A\modl)}(P^\bu,C^\bu)=0$ for all
$P^\bu\in\sD^\b(A\modl_\fp)$ vanishes in $\sD^\co(A\modl)$.
 Here one can argue as in~\cite[Lemma~2.2]{Kra}.
 Represent the object $C^\bu\in\sD^\co(A\modl)$ by a complex of
fp\+injective left $A$\+modules~$J^\bu$; then
$\Hom_{\sD^\co(A\modl)}(P^\bu,C^\bu[*])$ can be computed as
$\Hom_A(P^\bu,J^\bu)$.

 If the complex $J^\bu$ (or~$C^\bu$) has a nonzero cohomology module
$H^nJ^\bu\ne0$ in some degree~$n$, then there exists a finitely
presented $A$\+module $P$ (e.~g., $P=A$) and a morphism of complexes
$P\rarrow J^\bu[n]$ inducing a nonzero map of the cohomology modules.
 Otherwise, when $H^*(J^\bu)=0$, one has $H^{n+1}\Hom_A(P,J^\bu)\simeq
\Ext^1_A(P,Z^n)$, where $Z^n$ is the kernel of the differential
$J^n\rarrow J^{n+1}$, for any finitely presented left $A$\+module~$P$.
 If $\Ext^1_A(P,Z^n)=0$ for all $P$ and~$n$, then the $A$\+modules
$Z^n$ are fp\+injective and the complex $J^\bu$ is acyclic in
the exact category $A\modl_\fpi$.
 Since by the assumption of~(b) this exact category has finite
homological dimension, by~\cite[Remark~2.1]{Psemi} the complex $J^\bu$
is coacyclic (and even absolutely acyclic) in $A\modl_\fpi$.
 (Cf.~\cite[Proposition~6.4]{Sto}, where a similar argument is presented
without any homological dimension assumptions.)
\end{proof}

\Section{Dualizing Complexes and Contravariant Duality}
\label{contravariant-duality}

 The notion of a bimodule over an arbitrary pair of rings is
inherently problematic from the homological point of view.
 It suffices to consider the example of $A$\+$B$\+bimodules for
the pair of rings $A=\boZ/4$ and $B=\boZ/2$ in order to see where
the problem lies.
 In particular, it is \emph{not} always possible to embed
an $A$\+$B$\+module into an $A$\+injective $A$\+$B$\+bimodule.
 Restricting to the case $A=B$ does not help much, as the trouble
repeats itself for the ring $A=B=\boZ/2\oplus\boZ/4$.
 Assuming that both $A$ and $B$ are flat algebras over the same
commutative ring~$k$ and working with $A$\+$B$\+bimodules over~$k$
(i.~e., left modules over $A\ot_k B^\rop$) resolves the problem.

 One has to deal with this issue when defining the notion of
a dualizing complex over a pair of noncommutative rings.
 Several approaches have been tried in the literature, from restricting
outright to the case of algebras over a field~\cite{Yek,YZ} to
specifying explicit left and right adjustness conditions on complexes
of bimodules~\cite{CFH,Miy,Pcosh}.
 In this section we show that the most na\"\i ve weak definition of
a dualizing complex works well enough to provide a contravariant
equivalence between bounded derived categories of finitely presented
modules.

 Given two associative rings $A$ and $B$, denote by $A\bimod B$
the abelian category of $A$\+$B$\+bimodules.
 Let $A$ be a left coherent ring and $B$ be a right coherent ring.
 A complex of $A$\+$B$\+bimodules $D^\bu\in\sD(A\bimod B)$ is said to be
a \emph{weak dualizing complex} for $A$ and $B$ if the following
conditions are satisfied:
\begin{enumerate}
\renewcommand{\theenumi}{\roman{enumi}$^{\mathrm{w}}$}
\item as a complex of left $A$\+modules, $D^\bu$ is quasi-isomorphic
to a finite complex of injective $A$\+modules; and as a complex of
right $B$\+modules, $D^\bu$ is quasi-isomorphic to a finite complex
of injective $B$\+modules;
\renewcommand{\theenumi}{\roman{enumi}}
\item the $A$\+$B$\+bimodules of cohomology $H^*(D^\bu)$ of
the complex $D^\bu$ are finitely presented left $A$\+modules and
finitely presented right $B$\+modules;
\item the homothety maps $A\rarrow\Hom_{\sD(\modr B)}(D^\bu,D^\bu[*])$
and $B^\rop\rarrow\Hom_{\sD(A\modl)}\allowbreak(D^\bu,D^\bu[*])$ are
isomorphisms of graded rings.
\end{enumerate} 

 The following result is a slight generalization
of~\cite[Propositions~3.4\+-3.5]{Yek} and~\cite[first assertion of
Corollary~2.8]{Miy}.

\begin{thm}
 Let $D^\bu$ be a weak dualizing complex for a left coherent ring $A$
and a right coherent ring~$B$.
 Then there is an anti-equivalence between the bounded derived
categories\/ $\sD^\b(A\modl_\fp)$ and\/ $\sD^\b(\modrfp B)$ of finitely
presented left $A$\+modules and finitely presented right $B$\+modules
provided by the mutually inverse functors\/ $\boR\Hom_A({-},D^\bu)$
and\/ $\boR\Hom_{B^\rop}({-},D^\bu)$. 
\end{thm}

\begin{proof}
 Let $A\modl_\proj\subset A\modl$ denote the additive category of
projective left $A$\+modules and let $\Hot^-(A\modl_\proj)$ be its
bounded above homotopy category.
 Then the bounded derived category of left $A$\+modules
$\sD^\b(A\modl)$ can be identified with the full subcategory in
$\Hot^-(A\modl_\proj)$ consisting of complexes with bounded cohomology.
 Restricting the contravariant functor
$$
 \Hom_A({-},D^\bu)\:\Hot(A\modl)^\sop\lrarrow\Hot(\modr B)
$$
acting between the homotopy categories of left $A$\+modules and
right $B$\+modules to the full subcategory
$$
 \sD^\b(A\modl)\subset\Hot^-(A\modl_\proj)\subset\Hot(A\modl),
$$
we obtain the derived functor $\boR\Hom_A({-},D^\bu)\:
\sD^\b(A\modl)^\sop\rarrow\sD(\modr B)$.
 
 Let us show that the image of this functor is contained in
the bounded derived category $\sD^\b(\modr B)\subset\sD(\modr B)$
of right $B$\+modules.
 Indeed, the property of a complex of $B$\+modules to have bounded
cohomology only depends on its underlying complex of abelian groups.
 The composition $\sD^\b(A\modl)^\sop\rarrow\sD(\modr B)\rarrow\sD(\Ab)$
of the functor $\boR\Hom_A$ with the forgetful functor $\sD(\modr B)
\rarrow\sD(\Ab)$ to the derived category of abelian groups can be
computed as the functor $\Hom_A({-},{}'\!\.D^\bu)$, where ${}'\!\.D^\bu$
is a finite complex of injective left $A$\+modules quasi-isomorphic
to~$D^\bu$.
 The latter functor obviously takes $\sD^\b(A\modl)$ to
$\sD^\b(\Ab)$.

 We have obtained the right derived functor 
$\boR\Hom_A({-},D^\bu)\:\sD^\b(A\modl)^\sop\rarrow\sD^\b(\modr B)$;
similarly, there is the right derived functor
$\boR\Hom_{B^\rop}({-},D^\bu)\:\allowbreak
\sD^\b(\modr B)^\sop\rarrow\sD^\b(A\modl)$.
 Let us show that these two contravariant functors are right
adjoint to each other; in other words, for any complexes
$M^\bu\in\sD^\b(A\modl)$ and $N^\bu\in\sD^\b(\modr B)$ there is
a natural isomorphism of abelian groups {\hbadness=1600
$$
 \Hom_{\sD^\b(\modr B)}(N^\bu,\boR\Hom_A(M^\bu,D^\bu))\simeq
 \Hom_{\sD^\b(A\modl)}(M^\bu,\boR\Hom_{B^\rop}(N^\bu,D^\bu)).
$$
 Indeed, represent the object $M^\bu$ by a bounded above complex
of projective left $A$\+modules $P^\bu$ and the object $N^\bu$ by
a bounded above complex of projective right $B$\+modules $Q^\bu$;
then the passage to the degree-zero cohomology groups in
the natural isomorphism of complexes
$\Hom_{B^\rop}(Q^\bu,\Hom_A(P^\bu,D^\bu))\simeq
\Hom_A(P^\bu,\Hom_{B^\rop}(Q^\bu,D^\bu))$ provides the desired
isomorphism of the Hom groups in the derived categories.}

 Furthermore, let us check that the functor $\Hom_A({-},D^\bu)$ takes
the bounded derived category of finitely presented left $A$\+modules
$\sD^\b(A\modl_\fp)\subset\sD^\b(A\modl)$ into the bounded derived
category of finitely presented right $B$\+modules $\sD^\b(\modrfp B)
\subset\sD^\b(\modr B)$ (see Corollary~\ref{derived-fp-fully-faithful}).
 Indeed, an object of $\sD^\b(A\modl_\fp)$ can be represented by
a bounded above complex of finitely generated projective left
$A$\+modules $P^\bu$, and the complex of $A$\+$B$\+bimodules with
$D^\bu$ with bounded cohomology can be replaced by a quasi-isomorphic
finite complex of $A$\+$B$\+bimodules.
 Then the property of every cohomology module of the complex
$\Hom_A(P^\bu,D^\bu)$ to be finitely presented over $B$ only depends
on a finite fragment of the complex $P^\bu$, which reduces
the question to the case of a one-term complex $P^\bu=P$ corresponding
to a finitely generated projective $A$\+module~$P$.
 It remains to recall that the cohomology bimodules of
the complex~$D^\bu$ were assumed to be finitely presented right
$B$\+modules.

 We have constructed the derived functor
\begin{align*}
 \boR\Hom_A({-},D^\bu)\:\sD^\b(A\modl_\fp)^\sop &\lrarrow
 \sD^\b(\modrfp B); \\
\intertext{similarly one obtains the derived functor}
 \boR\Hom_{B^\rop}({-},D^\bu)\:\sD^\b(\modrfp B)^\sop &\lrarrow
 \sD^\b(A\modl_\fp).
\end{align*}
 It remains to prove that these are mutually inverse anti-equivalences.
 For this purpose, we will show that the adjunction maps
are quasi-isomorphisms; it suffices to check that these are
quasi-isomorphisms of complexes of abelian groups.

 Let an object of the derived category $\sD^\b(A\modl_\fp)$ be
represented by a bounded above complex of finitely generated
projective left $A$\+modules~$P^\bu$.
 Replace the complex $D^\bu$ by a quasi-isomorphic finite complex
of $A$\+$B$\+bimodules; and let ${}''\!\.D^\bu$ be a finite complex of
injective right $B$\+modules endowed with a quasi-isomorphism of
complexes of right $B$\+modules $D^\bu\rarrow{}''\!\.D^\bu$.
 Then checking that the natural map of complexes of abelian groups
$P^\bu\rarrow\Hom_{B^\rop}(\Hom_A(P^\bu,D^\bu),{}''\!\.D^\bu)$ is
a quasi-isomorphism reduces to the case of a one-term complex
$P^\bu=A$, when the desired assertion becomes an expression of
the assumption that the homothety map
$A\rarrow\Hom_{\sD(\modr B)}(D^\bu,D^\bu[*])$ is an isomorphism of
graded rings.
\end{proof}

 Notice that the construction of the duality functors between
the categories $\sD^\b(A\modl_\fp)$ and $\sD^\b(\modrfp B)$ becomes
much simpler when $D^\bu$ is a finite complex of fp\+injective
left $A$\+modules and right $B$\+modules (cf.\ the definition
of a dualizing complex in the next Section~\ref{covariant-duality}).
 In this case one can simply apply the underived functors
$\Hom_A({-},D^\bu)$ and $\Hom_{B^\rop}({-},D^\bu)$ to bounded complexes
of finitely presented left $A$\+modules and right $B$\+modules,
obtaining complexes of right $B$\+modules and left $A$\+modules
with bounded and finitely presented cohomology modules (which
form categories equivalent to $\sD^\b(\modrfp B)$ and
$\sD^\b(A\modl_\fp)$ by Corollary~\ref{derived-fp-fully-faithful}).

 The following definition can be found in~\cite{Miy}
(see also~\cite[Section~B.4]{Pcosh}).
 Let $A$ be a left coherent ring and $B$ be a right coherent ring.
 A finite complex of $A$\+$B$\+bimodules $D^\bu$ is said to be
a \emph{strong dualizing complex} for the rings $A$ and $B$ if
the following conditions are satisfied:
\begin{enumerate}
\renewcommand{\theenumi}{\roman{enumi}$^{\mathrm{s}}$}
\item the terms of the complex $D^\bu$ are injective left $A$\+modules
and injective right $B$\+modules;
\renewcommand{\theenumi}{\roman{enumi}}
\item the $A$\+$B$\+bimodules of cohomology $H^*(D^\bu)$ of
the complex $D^\bu$ are finitely presented left $A$\+modules and
finitely presented right $B$\+modules;
\renewcommand{\theenumi}{\roman{enumi}$^{\mathrm{s}}$}
\item the homothety maps $A\rarrow\Hom_{B^\rop}(D^\bu,D^\bu)$ and
$B^\rop\rarrow\Hom_A(D^\bu,D^\bu)$ are quasi-isomorphisms of
DG\+rings.
\end{enumerate}

 The condition~(ii) is the same as in the above definition of
a weak dualizing complex.
 In the assumption of the condition~(i$^{\mathrm{s}}$),
the condition~(iii$^{\mathrm{s}}$) is an equivalent restatement
of the condition~(iii).
 The following result is the second assertion
of~\cite[Corollary~2.8]{Miy}; see also~\cite[Proposition~1.3]{YZ}.

\begin{thm}
 Let $D^\bu$ be a strong dualizing complex for a left coherent
ring $A$ and a right coherent ring~$B$.
 Then there is an anti-equivalence between the derived categories
of unbounded complexes of left $A$\+modules and right $B$\+modules
with finitely presented cohomology modules\/ $\sD_\fp(A\modl)$ and\/
$\sD_\fp(\modr B)$ provided by the mutually inverse functors\/
$\Hom_A({-},D^\bu)$ and\/ $\Hom_{B^\rop}({-},D^\bu)$.
\end{thm}

\begin{proof}
 In the assumption~(i$^{\mathrm{s}}$), the functors
$\Hom_A({-},D^\bu)\:\Hot(A\modl)^\sop\rarrow\Hot(\modr B)$ and
$\Hom_{B^\rop}({-},D^\bu)\:\Hot(\modr B)^\sop\rarrow\Hot(A\modl)$ take
acyclic complexes to acyclic complexes and consequently induce
triangulated functors \hbadness=1100
\begin{align*}
 \Hom_A({-},D^\bu)\:\sD(A\modl)^\sop&\lrarrow\sD(\modr B) \\
\intertext{and}
 \Hom_{B^\rop}({-},D^\bu)\:\sD(\modr B)^\sop&\lrarrow\sD(A\modl).
\end{align*}
 It is not difficult to see that these two contravariant functors
are right adjoint to each other.
 Now verifying that these functors take the full subcategories
$\sD_\fp(A\modl)\subset\sD(A\modl)$ and $\sD_\fp(\modr B)\subset
\sD(\modr B)$ into each other and the adjunction morphisms are
quasi-isomorphisms for complexes from these subcategories
depends only on finite fragments of the complexes involved,
which makes these questions straightforward (and certainly easier
than the ones resolved in the previous proof).
\end{proof}

\begin{ex} \label{double-dualization-coalgebra-example}
 Let $K\rarrow A$ be an associative ring homomorphism such that $A$
is a finitely generated projective left $K$\+module.
 Set $\C=\Hom_K(A,K)$; then the map $\C\rarrow\C\ot_K\C$ dual to
the multiplication map $A\ot_KA\rarrow A$ endows the $K$\+$K$\+bimodule
$\C$ with the structure of a coassociative coring over~$K$.
 The $K$\+$K$\+bimodule $\C$ is a finitely generated projective right
$K$\+module by construction; suppose further that it is also
a finitely generated projective left $K$\+module.
 Set $B=\Hom_K(\C,K)$; then the map $B\ot_KB\rarrow B$ dual to
the comultiplication map $\C\rarrow\C\ot_K\C$ endows
the $K$\+$K$\+bimodule $B$ with the structure of an associative ring.
 The unit map/ring homomorphism $K\rarrow A$ is transformed into
a counit map $\C\rarrow K$ and into a unit map/ring homomorphism
$K\rarrow B$.

 The category of left $A$\+modules is isomorphic to the category of
left $\C$\+comodules, and the category of right $B$\+modules is
isomorphic to the category of right $\C$\+comodules.
 So, in particular, the natural $\C$\+$\C$\+bicomodule structure
on $\C$ can be viewed as an $A$\+$B$\+bimodule structure
(see~\cite[Section~10.1.1]{Psemi}; cf.~\cite[Section~3.2]{BP}).

 Assume that the ring $K$ is left and right coherent, and
suppose further that $K$ is a left and right Gorenstein ring, i.~e.,
$K$ has a finite injective dimension as a left and right module
over itself.
 Then the $K$\+$K$\+bimodule $K$ is a weak dualizing complex for
the rings $K$ and $K$ \cite[Example~2.3(a)]{YZ}, while
the $A$\+$B$\+bimodule $\C$ is a weak dualizing complex for
the rings $A$ and~$B$ (cf.~\cite[Section~5]{Yek}
and~\cite[Section~3.10]{Pkoszul}).
 When the ring $K$ is classically semisimple, the $A$\+$B$\+bimodule
$\C$ is even a strong dualizing complex for the rings $A$ and~$B$.
\end{ex}

\begin{ex}
 Let $K$ be a commutative ring and $A$ be an associative $K$\+algebra
(with unit).
 Assume that the ring $K$ is coherent and the ring $A$ is a finitely
generated projective $K$\+module.
 Let $D_K^\bu$ be a (weak or strong) dualizing complex for the ring
$K$, i.~e., a complex of $K$\+modules that, viewed as a complex of
$K$\+$K$\+bimodules, is a (weak or strong) dualizing complex for
the rings $K$ and~$K$.
 Then the complex of $A$\+$A$\+bimodules $\Hom_K(A,D_K^\bu)$ is
a (respectively, weak or strong) dualizing complex for the rings $A$
and $A$ (cf.~\cite[Example~3.8 and Corollary~5.6]{Yek}).
\end{ex}

\Section{Covariant Duality Theorem}  \label{covariant-duality}

 The aim of this section is to extend the noncommutative covariant
Serre--Grothendieck duality theory developed in
the papers~\cite{Jor,Kra,IK} from Noetherian to coherent rings.
 Here is our main definition in this context.

 Let $A$ be a left coherent ring and $B$ be a right coherent ring.
 A finite complex of $A$\+$B$\+bimodules $D^\bu$ is called
a \emph{dualizing complex} for the rings $A$ and $B$ if it satisfies
the following conditions:
\begin{enumerate}
\renewcommand{\theenumi}{\roman{enumi}}
\item the terms of the complex $D^\bu$ are fp\+injective left
$A$\+modules and fp\+injective right $B$\+modules;
\item the $A$\+$B$\+bimodules of cohomology $H^*(D^\bu)$ of
the complex $D^\bu$ are finitely presented left $A$\+modules and
finitely presented right $B$\+modules;
\item the homothety maps $A\rarrow\Hom_{\sD(\modr B)}(D^\bu,D^\bu[*])$ 
and $B^\rop\rarrow\Hom_{\sD(A\modl)}\allowbreak(D^\bu,D^\bu[*])$ are
isomorphisms of graded rings.
\end{enumerate}

 The conditions~(ii\+iii) are the same as in the definition of
a weak dualizing complex in Section~\ref{contravariant-duality}.
 The following lemma explains how the fp\+injectivity condition
in~(i) is to be used.
 For the rest of this section, we assume that the ring $A$ is
left coherent and the ring $B$ is right coherent.

\begin{lem}  \label{tensor-hom-flat-fp-injective}
\textup{(a)} Let $F$ be a flat left $B$\+module and $E$ be
an $A$\+fp\+injective $A$\+$B$\+bimodule.
 Then the tensor product $E\ot_BF$ is an fp\+injective left
$A$\+module. \par
\textup{(b)} Let $J$ be an injective left $A$\+module and
$E$ be a $B$\+fp\+injective $A$\+$B$\+bimodule.
 Then the left $B$\+module $\Hom_A(E,J)$ is flat.
\end{lem}

\begin{proof}
 In part~(a), one can check that the natural map
$\Hom_A(M,E)\ot_B F\rarrow \Hom_A(M\;E\ot_BF)$ is an isomorphism
for any finitely presented left $A$\+module~$M$.
 (Alternatively, one can use the Govorov--Lazard characterization
of flat modules as filtered inductive limits of finitely generated
projective ones together with the fact that the class of fp\+injective
left modules over a left coherent ring is closed under filtered
inductive limits.)
 In part~(b), one needs to show that the natural map
$N\ot_B\Hom_A(E,J)\rarrow\Hom_A(\Hom_{B^\rop}(N,E),J)$ is
an isomorphism for any finitely presented right $B$\+module~$N$.
\end{proof}

 The next lemma is a generalization of~\cite[Lemma~B.4.1]{Pcosh}.
 We assume that $D^\bu$ is a dualizing complex for the rings $A$
and~$B$.

\begin{lem}  \label{dualizing-covariant-quasi-isomorphisms}
\textup{(a)} Let $F$ be a flat left $B$\+module and $J^\bu$ be
a bounded below complex of injective left $A$\+modules endowed with
a quasi-isomorphism of complexes of left $A$\+modules
$D^\bu\ot_B F\rarrow J^\bu$.
 Then the natural morphism of complexes of left $B$\+modules
$F\rarrow\Hom_A(D^\bu,J^\bu)$ is a quasi-isomorphism. \par
\textup{(b)} Let $J$ be an injective left $A$\+module.
 Then the natural morphism of complexes of left $A$\+modules
$D^\bu\ot_B\Hom_A(D^\bu,J)\rarrow J$ is a quasi-isomorphism.
\end{lem}

\begin{proof}
 Part~(a): let ${}'\!\.D^\bu$ be a bounded above complex of finitely
generated projective left $A$\+modules endowed with a quasi-isomophism
of complexes of left $A$\+modules ${}'\!\.D^\bu\rarrow D^\bu$.
 Then the natural morphism $\Hom_A(D^\bu,J^\bu)\rarrow
\Hom_A({}'\!\.D^\bu,J^\bu)$ is a quasi-isomorphism of complexes of
abelian groups, as is the natural morphism $\Hom_A({}'\!\.D^\bu\;
D^\bu\ot_BF)\rarrow\Hom_A({}'\!\.D^\bu,J^\bu)$.
 The square of morphisms of complexes of abelian groups $F\rarrow
\Hom_A(D^\bu,J^\bu)\rarrow\Hom_A({}'\!\.D^\bu,J^\bu)$ and $F\rarrow
\Hom_A({}'\!\.D^\bu\;D^\bu\ot_BF)\rarrow\Hom_A({}'\!\.D^\bu,J^\bu)$ is
commutative, so it suffices to show that the morphism $F\rarrow
\Hom_A({}'\!\.D^\bu\;D^\bu\ot_BF)$ is a quasi-isomorphism.
 Now the complex of abelian groups $\Hom_A({}'\!\.D^\bu\;
D^\bu\ot_BF)$ is isomorphic to $\Hom_A({}'\!\.D^\bu,D^\bu)
\ot_BF$, and it remains to recall that the map $B\rarrow
\Hom_A({}'\!\.D^\bu,D^\bu)$ induced by the right action of $B$ in
$D^\bu$ is a quasi-isomorphism of complexes of right $B$\+modules by
the assumption~(iii).

 Part~(b): let ${}''\!\.D^\bu$ be a bounded above complex of finitely
generated projective right $B$\+modules endowed with a quasi-isomophism
of complexes of right $B$\+modules ${}''\!\.D^\bu\rarrow D^\bu$.
 Then the natural morphism ${}''\!\.D^\bu\ot_B\Hom_A(D^\bu,J)\rarrow
D^\bu\ot_B\Hom_A(D^\bu,J)$ is a quasi-isomorphism of complexes of
abelian groups, and it suffices to show that the composition
${}''\!\.D^\bu\ot_B\Hom_A(D^\bu,J)\rarrow D^\bu\ot_B\Hom_A(D^\bu,J)
\rarrow J$ is also a quasi-isomorphism of complexes of abelian groups.
 Now the complex of abelian groups ${}''\!\.D^\bu\ot_B\Hom_A(D^\bu,J)$
is isomorphic to $\Hom_A(\Hom_{B^\rop}({}''\!\.D^\bu,D^\bu),J)$, and
it remains to recall that the map $A\rarrow\Hom_{B^\rop}({}''\!\.D^\bu,
D^\bu)$ is a quasi-isomorphism of complexes of left $A$\+modules
by the assumption~(iii).
\end{proof}

 The following result generalizes the first assertion
of~\cite[Proposition~1.5]{CFH} (see also~\cite[Corollary~B.4.2]{Pcosh}).
 It can be used in conjunction with
Proposition~\ref{cardinality-generators-ideals}.

\begin{prop}  \label{cfh-generalized}
 Let $D^\bu$ be a dualizing complex for a left coherent ring $A$ and
a right coherent ring~$B$.
 Then the supremum of projective dimensions of flat left $B$\+modules
cannot exceed the supremum of injective dimensions of fp\+injective
left $A$\+modules by more than the length of the complex~$D^\bu$.
\end{prop}

\begin{proof}
 Assume that the complex $D^\bu$ is concentrated in the cohomological
degrees from~$i$ to $i+d$, and the supremum of injective
dimensions of fp\+injective left $A$\+modules is a finite integer~$n$.
 In order to show that the projective dimension of any flat left
$B$\+module does not exceed $n+d$, it suffices to check that
$\Ext_B^{n+d+1}(F,G)=0$ for any flat left $B$\+modules $F$ and~$G$.
 Let $P_\bu$ be a left projective resolution of the $B$\+module $F$
and $J^\bu$ be a complex of injective left $A$\+modules concentrated
in the cohomological degrees from~$i$ to $i+n+d$ and endowed with
a quasi-isomorphism of complexes of left $A$\+modules $D^\bu\ot_B G
\rarrow J^\bu$.
 By Lemma~\ref{dualizing-covariant-quasi-isomorphisms}(a),
the natural map of complexes of abelian groups
$$
 \Hom_B(P_\bu,G)\lrarrow\Hom_B(P_\bu,\Hom_A(D^\bu,J^\bu))
$$
is a quasi-isomorphism.
 The right-hand side is isomorphic to the complex $\Hom_A(D^\bu\ot_B
P_\bu\;J^\bu)$, which is quasi-isomorphic to $\Hom_A(D^\bu\ot_BF\;J^\bu)$.
 The latter complex is concentrated in the degrees from~$-d$ to~$n+d$.
\hbadness=1200
\end{proof}

 Denote by $B\modl_\fl$ the exact category of flat left $B$\+modules
and by $B\modl_\proj$ the additive category of projective left
$B$\+modules.
 The following results, resembling Theorems~\ref{coderived-fp-injective}
and~\ref{coderived-homotopy-injectives}, were established in our
previous papers.

\begin{thm}  \label{contraderived-flat-projectives}
 Let $B$ be a right coherent ring.  Then \par
\textup{(a)} the triangulated functor between the contraderived
categories\/ $\sD^\ctr(B\modl_\fl)\rarrow\sD^\ctr(B\modl)$ induced by
the embedding of exact categories $B\modl_\fl\rarrow B\modl$ is
an equivalence of triangulated categories; \par
\textup{(b)} assuming that flat left $B$\+modules have finite
projective dimensions, the triangulated functor\/
$\Hot(B\modl_\proj)\rarrow\sD^\ctr(B\modl)$ induced by the embedding
of additive/exact categories $B\modl_\proj\rarrow B\modl$ is
an equivalence of triangulated categories.
\end{thm}

\begin{proof}
 Notice that the class of flat left modules over a right coherent ring
is closed under infinite products, so the triangulated functor in~(a)
is well-defined.
 The assertion of part~(a) is provided by~\cite[Remark~1.5]{Psing}
and/or~\cite[Proposition~A.3.1(b)]{Pcosh}.
 Part~(b) was proven in~\cite[Section~3.8]{Pkoszul}.
 In both cases, the same assertions also hold for the contraderived
category of left CDG\+modules over a CDG\+ring $(B,d,h)$ with
a right graded coherent underlying graded ring $B$, the contraderived
category of CDG\+modules with flat underlying graded modules, and
the homotopy category of CDG\+modules with projective underlying
graded modules.
\end{proof}

 The next theorem, generalizing~\cite[Theorem~4.8]{IK}
(see also~\cite[Corollary~B.4.10 and Theorem~D.2.5]{Pcosh}),
is one of the most important results of this paper.

\begin{thm}  \label{derived-co-contra}
 Let $D^\bu$ be a dualizing complex for a left coherent ring $A$ and
a right coherent ring~$B$.
 Assume that fp\+injective left $A$\+modules have finite injective
dimensions.
 Then there is an equivalence between the coderived category of
left $A$\+modules\/ $\sD^\co(A\modl)$ and the contraderived category
of left $B$\+modules\/ $\sD^\ctr(B\modl)$ provided by the mutually
inverse functors\/ $\boR\Hom_A(D^\bu,{-})$ and\/ $D^\bu\ot^\boL_B{-}$.
\end{thm}

\begin{proof}
 The derived functor $\boR\Hom_A(D^\bu,{-})\:\sD^\co(A\modl)\rarrow
\sD^\ctr(B\modl)$ is constructed by identifying the coderived
category $\sD^\co(A\modl)$ with the homotopy category
$\Hot(A\modl_\inj)$ and applying the functor $\Hom_A(D^\bu,{-})$
to complexes of injective left $A$\+modules.
 By Lemma~\ref{tensor-hom-flat-fp-injective}(b), this produces
complexes of flat left $B$\+modules.

 The derived functor $D^\bu\ot^\boL_B{-}\:\sD^\ctr(B\modl)\rarrow
\sD^\co(A\modl)$ is obtained by identifying the contraderived
category $\sD^\ctr(B\modl)$ of arbitrary left $B$\+modules with
the contraderived category $\sD^\ctr(B\modl_\fl)$ of flat left
$B$\+modules and applying the functor $D^\bu\ot_B{-}$ to
complexes of flat left $B$\+modules.
 By Lemma~\ref{tensor-hom-flat-fp-injective}(a), this produces
complexes of fp\+injective left $A$\+modules.
 Since the category of flat left $B$\+modules has finite homological
dimension by the assumption of Theorem in view of
Proposition~\ref{cfh-generalized}, any contraacyclic complex of
flat left $B$\+modules is absolutely acyclic with respect to
the exact category of flat left $B$\+modules, and the functor
$D^\bu\ot_B{-}$ transforms such a complex into an absolutely acyclic
complex of fp\+injective left $A$\+modules; so the derived
functor $D^\bu\ot^\boL_B{-}$ is well-defined.

 It is clear that the derived functor $D^\bu\ot^\boL_B{-}$ is left
adjoint to the derived functor $\boR\Hom_A(D^\bu,{-})$, so it remains
to show that the adjunction morphisms are isomorphisms in
$\sD^\co(A\modl)$ and $\sD^\ctr(B\modl)$.
 Let $J^\bu$ be a complex of injective left $A$\+modules.
 Then by Lemma~\ref{dualizing-covariant-quasi-isomorphisms}(b)
the cone of the morphism of complexes of fp\+injective left
$A$\+modules $D^\bu\ot_B\Hom_A(D^\bu,J^\bu)\rarrow J^\bu$ is the total
complex of a finite acyclic complex of complexes of fp\+injective left
$A$\+modules, hence a coacyclic (and even an absolutely acyclic)
complex of fp\+injective left $A$\+modules.

 Let $F^\bu$ be a complex of flat left $B$\+modules.
 Consider the bicomplex of fp\+injective left $A$\+modules $D^j\ot_BF^i$
and pick a bicomplex of injective left $A$\+modules $J^{ij}$ together
with a morphism of bicomplexes of $A$\+modules $D^\bu\ot_BF^\bu\rarrow
J^{\bu\bu}$ such that the bicomplex $J^{\bu\bu}$ is concentrated in
a finite interval of gradings~$j$ and for every index~$i$ the morphism
of complexes $D^\bu\ot_B F^i\rarrow J^{i,\bu}$ is a quasi-isomorphism.
 To obtain such a bicomplex $J^{\bu\bu}$ one can, e.~g., use a functorial
injective resolution construction in the category of left $A$\+modules,
or a construction of embedding of an arbitrary bicomplex in an abelian
category with enough injectives into a bicomplex of injective objects.
 Denote simply by $\Hom_A(D^\bu,J^{\bu\bu})$ the totalization of
the tricomplex $\Hom_A(D^k,J^{ij})$ along the pair of indices $(j,k)$;
then, by Lemma~\ref{dualizing-covariant-quasi-isomorphisms}(a),
the natural morphism of bicomplexes of flat left $B$\+modules
$F^\bu\rarrow\Hom_A(D^\bu,J^{\bu\bu})$ is a quasi-isomorphism of
finite (and uniformly bounded) complexes at every fixed degree~$i$.
 Now the cone of the natural morphism between the totalizations of
the bicomplexes $D^\bu\ot_BF^\bu$ and $J^{\bu\bu}$ is an absolutely
acyclic complex of fp\+injective left $A$\+modules, and the cone
of the natural morphism from the complex $F^\bu$ to the total
complex of $\Hom_A(D^\bu,J^{\bu\bu})$ is an absolutely acyclic
complex of flat left $B$\+modules.
\end{proof}

\Section{Semiderived Categories and Relative Dualizing Complexes}
\label{semiderived-categories}

 The aim of this section is to define the notion of a \emph{relative
dualizing complex} for a pair of homomorphisms of noncommutative
rings $A\rarrow R$ and $B\rarrow S$, and obtain a related covariant
equivalence between the \emph{semiderived categories} of modules.
 One can say that a relative dualizing complex is ``a dualizing
complex in the direction of $A$ and $B$, and a dedualizing complex
in the direction of $R$ relative to $A$ and $S$ relative to $B$''
(see the paper~\cite{Pmgm} for a discussion of dedualizing complexes).
 The resulting equivalence of semiderived categories resembles
the \emph{derived semimodule-semicontramodule correspondence}
of~\cite[Sections~0.3.7 and~6.3]{Psemi}.

 The following definition is to be compared with those
in~\cite[Sections~0.3.3 and~2.3]{Psemi}; see
also~\cite[Proposition~2.1.1(2)]{Bec}.
 Let $A\rarrow R$ be a morphism of associative rings.
 The \emph{$R/A$\+semicoderived category} $\sD^\sico_A(R\modl)$
of left $R$\+modules is defined as the quotient category of
the homotopy category of left $R$\+modules $\Hot(R\modl)$ by its
thick subcategory of complexes of $R$\+modules \emph{that are
coacyclic as complexes of $A$\+modules}.
 Similarly, assuming that the ring $A$ is left coherent,
the $R/A$\+semicoderived category of $A$\+fp\+injective
left $R$\+modules $\sD^\sico_A(R\modl_{A\dfpi})$ is defined as
the quotient category of the homotopy category of (complexes of)
$A$\+fp\+injective left $R$\+modules $\Hot(R\modl_{A\dfpi})$ by
its thick subcategory of complexes that are \emph{coacyclic in
the exact category of fp\+injective left $A$\+modules}.

 The next definition is to be compared with those
in~\cite[Sections~0.3.6 and~4.3]{Psemi}; see
also~\cite[Proposition~2.1.1(1)]{Bec}.
 Let $B\rarrow S$ be a morphism of associative rings.
 The \emph{$S/B$\+semicontraderived category} $\sD^\sictr_B(S\modl)$
of left $S$\+modules is defined as the quotient category of
the homotopy category of left $S$\+modules $\Hot(S\modl)$ by its
thick subcategory of complexes of $S$\+modules \emph{that are
contraacyclic as complexes of $B$\+modules}.
 Similarly, assuming that the ring $B$ is right coherent,
the $S/B$\+semicontraderived category of $B$\+flat left
$S$\+modules $\sD^\sictr_B(S\modl_{B\dfl})$ is defined as
the quotient category of the homotopy category of (complexes of)
$B$\+flat left $S$\+modules $\Hot(S\modl_{B\dfl})$ by its
thick subcategory of complexes that are \emph{contraacyclic in
the exact category of flat left $B$\+modules}.

\begin{thm}  \label{semiderived-fp-injective-flat}
\textup{(a)} Let $A\rarrow R$ be a morphism of associative rings;
assume that the ring $A$ is left coherent and the ring $R$ is
a flat right $A$\+module.
 Then the triangulated functor between the semicoderived categories
$\sD^\sico_A(R\modl_{A\dfpi})\rarrow\sD^\sico_A(R\modl)$ induced
by the embeddings of exact categories $A\modl_\fpi\rarrow A\modl$
and $R\modl_{A\dfpi}\rarrow R\modl$ is an equivalence of
triangulated categories. \par
\textup{(b)} Let $B\rarrow S$ be a morphism of associative rings;
assume that the ring $B$ is right coherent and the ring $S$ is
a flat left $B$\+module.
 Then the triangulated functor between the semicontraderived categories 
$\sD^\sictr_B(S\modl_{B\dfl})\rarrow\sD^\sictr_B(S\modl)$ induced
by the embeddings of exact categories $B\modl_\fl\rarrow B\modl$
and $S\modl_{B\dfl}\rarrow S\modl$ is an equivalence of
triangulated categories.
\end{thm}

\begin{proof}
 Part~(a): in view of~\cite[Lemma~1.6]{Pkoszul}, it suffices to
show that for any complex of left $R$\+modules $M^\bu$ there
exists a complex of $A$\+fp\+injective left $R$\+modules $J^\bu$
together with a morphism of complexes of $R$\+modules $M^\bu\rarrow
J^\bu$ with a cone coacyclic as a complex of left $A$\+modules.
 Indeed, since $R$ is a flat right $A$\+module, any injective
left $R$\+module is also an injective left $A$\+module.
 Hence any complex of left $R$\+modules $M^\bu$ can be embedded into
a complex of $A$\+fp\+injective left $R$\+modules.
 The quotient complex can be also similarly embedded, etc.
 Totalizing the bicomplex constructed in this way by taking
infinite direct sums along the diagonals, one obtains the desired
complex of $A$\+fp\+injective left $R$\+modules~$J^\bu$;
the cone of the natural morphism $M^\bu\rarrow J^\bu$ is even coacyclic
as a complex of left $R$\+modules~\cite[Lemma~2.1]{Psemi}.

 Part~(b): it suffices to show that for any complex of left $S$\+modules
$N^\bu$ there exists a complex of $B$\+flat left $S$\+modules $F^\bu$
together with a morphism of complexes of $S$\+modules $F^\bu\rarrow
N^\bu$ with a cone contraacyclic as a complex of left $B$\+modules.
 Indeed, since $S$ is a flat left $B$\+module, any flat left
$S$\+module is also flat over~$B$.
 Hence any complex of left $S$\+modules $N^\bu$ is the image of
a surjective morphism from a complex of $B$\+flat left $S$\+modules.
 The kernel can be also presented as such an image, etc.
 Totalizing the bicomplex constructed in this way by taking
infinite products along the diagonals, one obtains the desired
complex of $B$\+flat left $S$\+modules~$F^\bu$;
the cone of the natural morphism $F^\bu\rarrow N^\bu$ is even
contraacyclic as a complex of left
$S$\+modules~\cite[Section~A.3]{Pcosh}.
\end{proof}

 Let us denote by $\sD(R\modl_{A\dinj})$ the quotient category of
the homotopy category of (complexes of) $A$\+injective left
$R$\+modules by the thick subcategory of complexes that are
\emph{contractible as complexes of left $A$\+modules}.
 Similarly, denote by $\sD(S\modl_{B\dproj})$ the quotient
category of the homotopy category of $B$\+projective left
$R$\+modules by the thick subcategory of complexes that are
\emph{contractible as complexes of left $B$\+modules}.
 Notice that the triangulated categories $\sD(R\modl_{A\dinj})$
and $\sD(S\modl_{B\dproj})$ are the \emph{conventional} derived
categories of the exact categories of $A$\+injective left
$R$\+modules and $B$\+projective left $S$\+modules $R\modl_{A\dinj}$
and $S\modl_{B\dproj}$ (the ``coderived category along~$A$'' and
the ``contraderived category along~$B$'' tokens are expressed
in the passages from the abelian category $R\modl$ to its exact
subcategory $R\modl_{A\dinj}$ and from the abelian category
$S\modl$ to its exact subcategory $S\modl_{B\dproj}$).

 The following result provides an interpretation of the semiderived
categories $\sD^\sico_A(R\modl)$ and $\sD^\sictr_B(R\modl)$ in
the spirit of the definitions of the coderived and contraderived
categories as ``homotopy categories of complexes of injectives''
and ``homotopy categories of complexes of projectives''
(as in~\cite{Jor,Kra,IK} and~\cite{Bec,Sto}).

\begin{thm}  \label{semiderived-injective-projective}
\textup{(a)} Let $A\rarrow R$ be a morphism of associative rings;
assume that the ring $A$ is left coherent, the ring $R$ is a flat
right $A$\+module, and all fp\+injective left $A$\+modules have
finite injective dimensions.
 Then the triangulated functor between the (semico)derived categories\/
$\sD(R\modl_{A\dinj})\rarrow\sD^\sico_A(R\modl)$ induced by
the embeddings of exact categories $A\modl_\inj\rarrow A\modl$ and
$R\modl_{A\dinj}\rarrow R\modl$ is an equivalence of triangulated
categories. \par
\textup{(b)} Let $B\rarrow S$ be a morphism of associative rings;
assume that the ring $B$ is right coherent, the ring $S$ is
a projective left $B$\+module, and all flat left $B$\+modules have
finite projective dimensions.
 Then the triangulated functor between the (semicontra)derived
categories\/ $\sD(S\modl_{B\dproj})\rarrow\sD^\sictr_B(S\modl)$
induced by the embeddings of exact categories $B\modl_\proj\rarrow
B\modl$ and $S\modl_{B\dproj}\rarrow S\modl$ is an equivalence of
triangulated categories.
\end{thm}

\begin{proof}
 Part~(a): in view of the construction in the proof of
Theorem~\ref{semiderived-fp-injective-flat}(a), it remains to show
that for any complex of $A$\+fp\+injective left $R$\+modules $J^\bu$
there exists a complex of $A$\+injective left $R$\+modules $K^\bu$
together with a morphism of complexes of $R$\+modules $J^\bu\rarrow
K^\bu$ with a cone coacyclic as a complex of $A$\+modules.
 This is easily done using the finite resolution argument
of~\cite[Sections~3.6\+-3.7]{Pkoszul} and~\cite[Section~A.5]{Pcosh}.
 The proof of part~(b) is similar (cf.~\cite[Section~3.8]{Pkoszul}).
 One only has to notice that since the ring $S$ is a projective left
$B$\+module, any projective left $S$\+module is also projective
over~$B$; so any complex of left $S$\+modules is the image of
a surjective morphism from a complex of $B$\+projective left
$S$\+modules.
\end{proof}

 In order to formulate the derived semico-semicontra correspondence
(noncommutative covariant relative Serre--Grothendieck duality)
theorem, we need several more definitions.
 Let $A\rarrow R$ be a morphism of associative rings; assume that
the ring $R$ is a flat right $A$\+module.
 A left $R$\+module $P$ is said to be \emph{weakly projective
relative to~$A$} (\emph{weakly $R/A$\+projective}) if the functor
$\Hom_R(P,{-})$ takes short exact sequences of $A$\+injective
left $R$\+modules to short exact sequence of abelian groups
(cf.~\cite[Sections~4.1 and~4.3]{BP} and~\cite[Sections~5.1.4, 5.3
and~9.1]{Psemi}).
 Similarly, let $B\rarrow S$ be a morphism of associative rings;
assume that the ring $S$ is a flat left $B$\+module.
 A right $S$\+module $F$ is said to be \emph{weakly flat relative
to~$B$} (\emph{weakly $S/B$\+flat}) if the functor $F\ot_S{-}$
takes short exact sequences of $B$\+flat left $S$\+modules to
short exact sequences of abelian groups
(cf.~\cite[Section~5.1.6]{Psemi}).

\begin{lem}  \label{weakly-relatively-adjusted}
\textup{(a)} A left $R$\+module $P$ is weakly $R/A$\+projective
if and only if\/ $\Ext_R^1(P,J)=0$ for any $A$\+injective left
$R$\+module $J$, and if and only if\/ $\Ext_R^n(P,J)=0$ for all\/ $n>0$
and any such~$J$.
 Consequently, the class of weakly $R/A$\+projective left
$R$\+modules is closed under extensions and the passages to
the kernels of surjective morphisms. \par
\textup{(b)} A right $S$\+module $F$ is weakly $S/B$\+flat if and
only if\/ $\Tor^R_1(F,G)=0$ for any $B$\+flat left $S$\+module $G$, and
if and only if\/ $\Tor^R_n(F,G)=0$ for all\/ $n>0$ and any such~$G$.
 Consequently, the class of weakly $S/B$\+flat right $S$\+modules
is closed under extensions and the passages to the kernels of
surjective morphisms.
\end{lem}

\begin{proof}
 Part~(a): it is clear that any left $R$\+module $P$ satisfying
the $\Ext^1$ vanishing condition satisfies the definition of
weakly relative projectivity.
 In order to show that $\Ext_R^{>0}(P,J)=0$ for any weakly
$R/A$\+projective left $R$\+module $P$ and $A$\+injective
left $R$\+module $J$, one simply notices that any injective right
resolution of the $R$\+module $J$ is exact with respect to
the exact category of $A$\+injective left $R$\+modules (since
injective left $R$\+modules are $A$\+injective).
 Part~(b): it is clear that any right $S$\+module $F$ satisfying
the $\Tor_1$ vanishing condition satisfies the definition of
weakly relative flatness.
 To check that $\Tor^R_{>0}(F,G)=0$ for any weakly $S/B$\+flat
right $S$\+module $F$ and any $B$\+flat left $S$\+module $G$, one
notices that any flat left resolution of the $S$\+module $G$ is
exact with respect to the exact category of $B$\+flat left
$S$\+modules (since flat left $S$\+modules are $B$\+flat). 
\end{proof}

 Here is the main definition of this section.
 Let $A\rarrow R$ and $B\rarrow S$ be a pair of associative ring
homomorphisms; assume that the ring $A$ is left coherent,
the ring $B$ is right coherent, the ring $R$ is a flat right
$A$\+module, and the ring $S$ is a flat left $B$\+module.
 A \emph{relative dualizing complex} for the pair of morphisms
$A\rarrow R$ and $B\rarrow S$ is a triple consisting of a dualizing
complex $D^\bu$ for the rings $A$ and $B$, a finite complex of
$R$\+$S$\+bimodules $T^\bu$, and a morphism of complexes of
$A$\+$B$\+bimodules $D^\bu\rarrow T^\bu$ satisfying the following
conditions:
\begin{enumerate}
\renewcommand{\theenumi}{\roman{enumi}}
\setcounter{enumi}{3}
\item the terms of the complex $T^\bu$ are weakly $R/A$\+projective
left $R$\+modules and weakly $S/B$\+flat right $S$\+modules;
\item the morphism of complexes of $R$\+$B$\+bimodules
$R\ot_AD^\bu\rarrow T^\bu$ and the morphism of complexes of
$A$\+$S$\+bimodules $D^\bu\ot_BS\rarrow T^\bu$ induced by
the morphism $D^\bu\rarrow T^\bu$ are quasi-isomorphisms of
finite complexes.
\end{enumerate}

\begin{ex}  \label{flat-tensor-example}
 (1)~Let $A$ and $B$ be associative algebras over a commutative
ring~$k$; assume that the ring $A$ is left coherent and the ring $B$
is right coherent.
 Suppose that a dualizing complex $D^\bu$ for the rings $A$ and $B$
is a complex of $A$\+$B$\+bimodules over~$k$ (i.~e., the left
and right $k$\+module structures on $D^\bu$ coincide).
 Let $U$ be a $k$\+flat associative algebra over~$k$.
 Consider the natural homomorphisms of associative rings
$A\rarrow U\ot_kA=R$ and $B\rarrow U\ot_kB=S$.
 Then the complex of $R$\+$S$\+bimodules $T^\bu=U\ot_k D^\bu$ together
with the natural morphism $D^\bu\rarrow T^\bu$ is a relative dualizing
complex for the pair of ring homomorphisms $A\rarrow R$
and $B\rarrow S$.

 (2)~In particular, let $A$ be a coherent commutative ring and
$R$ be an $A$\+flat associative $A$\+algebra.
 Let $D^\bu$ be a dualizing complex for the ring $A$ (i.~e.,
a complex of $A$\+modules that, viewed as a complex of
$A$\+$A$\+bimodules, is a dualizing complex for the rings $A$ and~$A$).
 Then the complex of $R$\+$R$\+bimodules $T^\bu=R\ot_AD^\bu$ together
with the natural morphism $D^\bu\rarrow T^\bu$ is a relative dualizing
complex for the pair of ring homomorphisms $A\rarrow R$ and
$A\rarrow R$.
\end{ex}

\begin{ex}
 (1)~Let $A\rarrow R$ be a homomorphism of associative algebras over
a field~$k$ such that the algebra $A$ is finite-dimensional and
the ring $R$ is a projective right $A$\+module.
 Then the tensor product $R\ot_AA^*$ has a natural structure of
a semialgebra over the coalgebra $A^*$ \cite[Section~10.2.1]{Psemi}.
 The $A$\+$A$\+bimodule $R\ot_AA^*$ is an injective right $A$\+module
by construction; suppose further that it is an injective left
$A$\+module.
 Set $S=\Hom_A(A^*\;R\ot_AA^*)$; the $A$\+$A$\+bimodule $S$ can be
also defined as the cotensor product $A\oc_{A^*}(R\ot_AA^*)$ over
the coalgebra~$A^*$.
 Then there is a natural associative algebra structure on $S$ and
a natural homomorphism of associative algebras $A\rarrow S$
\cite[Section~B.2.2]{Psemi}.
 The $R$\+$S$\+bimodule $R\ot_AA^*=\bT\simeq A^*\ot_AS$, together with
the natural map $A^*\rarrow\bT$, is a relative dualizing complex for
the pair of ring homomorphisms $A\rarrow R$ and $A\rarrow S$.

 (2)~More generally, let $K$ be a classically semisimple ring and
$K\rarrow A\rarrow R$ be associative ring homomorphisms such that
$A$ is a finitely generated projective left $K$\+module and $R$ is
a projective right $A$\+module.
 Set $\C=\Hom_K(A,K)$ and suppose that $\C$ is a finitely generated
projective left $K$\+module (see
Example~\ref{double-dualization-coalgebra-example} above).
 Set $B=\Hom_K(\C,K)$.
 The $A$\+$B$\+bimodule $R\ot_A\C$ is an injective right $B$\+module
by construction; suppose further that it is an injective left
$A$\+module.
 Set $S=\Hom_A(\C\;R\ot_A\C)$; the $B$\+$B$\+bimodule $S$ can be
also defined as the cotensor product $B\oc_\C(R\ot_A\C)$ over
the coring~$\C$.

 The tensor product $R\ot_A\C$ has a natural structure of a semialgebra
over the coring~$\C$.
 There is a natural associative ring structure on $S$ and a natural
homomorphism of associative rings $B\rarrow S$
\cite[Section~3.3]{BP}.
 The $R$\+$S$\+bimodule $R\ot_A\C=\bT\simeq\C\ot_BS$, together with
the natural map $\C\rarrow\bT$, is a relative dualizing complex for
the pair of ring homomorphisms $A\rarrow R$ and $B\rarrow S$.
\end{ex}

 The following theorem is our main result.

\begin{thm}  \label{relative-co-contra}
  Let $A\rarrow R$ and $B\rarrow S$ be a pair of associative ring
homomorphisms; assume that the ring $A$ is left coherent,
the ring $B$ is right coherent, the ring $R$ is a flat right
$A$\+module, the ring $S$ is a flat left $B$\+module, and all
fp\+injective left $A$\+modules have finite injective dimensions.
 Let $D^\bu\rarrow T^\bu$ be a relative dualizing complex for the pair
of morphisms $A\rarrow R$ and $B\rarrow S$.
 Then there is an equivalence between the $R/A$\+semicoderived
category of left $R$\+modules\/ $\sD^\sico_A(R\modl)$ and
the $S/B$\+semicontraderived category of left $S$\+modules\/
$\sD^\sictr_B(S\modl)$ provided by the mutually inverse functors\/
$\boR\Hom_R(T^\bu,{-})$ and\/ $T^\bu\ot^\boL_S{-}$.
\end{thm}

\begin{proof}
 The derived functor $\boR\Hom_R(T^\bu,{-})\:\sD^\sico_A(R\modl)\rarrow
\sD^\sictr_B(S\modl)$ is constructed by identifying the semicoderived
category $\sD^\sico_A(R\modl)$ with the derived category
$\sD(R\modl_{A\dinj})$ (see
Theorem~\ref{semiderived-injective-projective}(a)) and applying
the functor $\Hom_R(T^\bu,{-})$ to complexes of $A$\+injective
left $R$\+modules.
 Given a complex of $A$\+injective left $R$\+modules $J^\bu$, there is
a natural morphism of complexes of left $B$\+modules
$$
 \Hom_R(T^\bu,J^\bu)\lrarrow\Hom_R(R\ot_A D^\bu\;J^\bu)\simeq
 \Hom_A(D^\bu,J^\bu);
$$
in view of Lemma~\ref{weakly-relatively-adjusted}(a), the cone of
this morphism is the total complex of a finite acyclic complex
of complexes of left $B$\+modules, that is an absolutely acyclic
complex of left $B$\+modules.
 In particular, it follows that the complex of left $S$\+modules
$\Hom_R(T^\bu,J^\bu)$ is $B$\+contraacyclic whenever a complex of
$A$\+injective left $R$\+modules $J^\bu$ is contractible as
a complex of left $A$\+modules.

 The derived functor $T^\bu\ot^\boL_S{-}\:\sD^\sictr_B(S\modl)\rarrow
\sD^\sico_A(R\modl)$ is obtained by identifying the semicontraderived
category $\sD^\sictr_B(S\modl)$ with the semicontraderived category
$\sD^\sictr_B(S\modl_{B\dfl})$ (see
Theorem~\ref{semiderived-fp-injective-flat}(b)) and applying the functor
$T^\bu\ot_S{-}$ to complexes of $B$\+flat left $S$\+modules.
 Notice that, by Proposition~\ref{cfh-generalized}, the exact
category of flat left $B$\+modules has finite homological dimension,
so any contraacyclic complex of flat left $B$\+modules is absolutely
acyclic with respect to the exact category $B\modl_\fl$.
 Given a complex of $B$\+flat left $S$\+modules $G^\bu$, there is
a natural morphism of complexes of left $A$\+modules
$$
 D^\bu\ot_B G^\bu\simeq (D^\bu\ot_B S)\ot_S G^\bu\lrarrow
 T^\bu\ot_S G^\bu;
$$
in view of Lemma~\ref{weakly-relatively-adjusted}(b), the cone of
this morphism is the total complex of a finite acyclic complex of
complexes of left $A$\+modules, that is an absolutely acyclic
complexes of left $A$\+modules.
 It follows that the complex of left $R$\+modules $T^\bu\ot_SG^\bu$
is $A$\+coacyclic whenever a complex of $B$\+flat left $S$\+modules
$G^\bu$ is $B$\+contraacyclic.

 We have constructed the derived functors $\boR\Hom_R(T^\bu,{-})\:
\sD^\sico_A(R\modl)\rarrow\sD^\sictr_B(S\modl)$ and
$T^\bu\ot^\boL_S{-}\:\sD^\sictr_B(S\modl)\rarrow\sD^\sico_A(R\modl)$.
 It is easy to show that the former of them is right adjoint to
the latter (see, e.~g., \cite[Lemma~8.3]{Psemi}).
 It remains to check that the adjunction morphisms are isomorphisms
in $\sD^\sico_A(R\modl)$ and $\sD^\sictr_B(S\modl)$.
 Now we recall that, as we have just seen, the forgetful functors
\begin{align*}
 \sD^\sico_A(R\modl)&\lrarrow\sD^\co(A\modl) \\
 \intertext{and}
 \sD^\sictr_B(S\modl)&\lrarrow\sD^\ctr(B\modl)
\end{align*}
transform the functor $\boR\Hom_R(T^\bu,{-})$ into the functor
$\boR\Hom_A(D^\bu,{-})$ and the functor $T^\bu\ot^\boL_S{-}$ into
the functor $D^\bu\ot^\boL_B{-}$, i.~e., there are commutative
diagrams of triangulated functors
$$\dgARROWLENGTH=4em
\!\!\!
\begin{diagram} 
\node{\sD^\sico_A(R\modl)}
\arrow[2]{e,t}{\boR\Hom_R(T^\bu,{-})}
\arrow{s}
\node[2]{\sD^\sictr_B(S\modl)} \arrow{s} \\
\node{\sD^\co(A\modl)}
\arrow[2]{e,t}{\boR\Hom_A(D^\bu,{-})}
\node[2]{\sD^\ctr(B\modl)}
\end{diagram}
\,
\begin{diagram} 
\node{\sD^\sico_A(R\modl)}
\arrow{s}
\node[2]{\sD^\sictr_B(S\modl)} \arrow{s}
\arrow[2]{w,t}{T^\bu\ot^\boL_S{-}} \\
\node{\sD^\co(A\modl)}
\node[2]{\sD^\ctr(B\modl)}
\arrow[2]{w,t}{D^\bu\ot^\boL_B{-}}
\end{diagram}
\!\!\!
$$
 The forgetful functors also transform the adjunction morphisms
for the pair of functors $\boR\Hom_R(T^\bu,{-})$ and
$T^\bu\ot^\boL_S{-}$ into the adjunction morphisms for the pair of
functors $\boR\Hom_A(D^\bu,{-})$ and $D^\bu\ot^\boL_B{-}$.
 Since the latter pair of adjunction morphisms are isomorphisms
by Theorem~\ref{derived-co-contra} and the forgetful functors are
conservative, the former pair of adjunction morphisms are
isomorphisms, too.
\end{proof}

\Section{The Semitensor Product}

 The coderived category of quasi-coherent sheaves on a separated
Noetherian scheme with a dualizing complex has a tensor category
structure depending on the choice of a dualizing complex, which is
the unit object of this tensor category structure (see~\cite[Chapter~6,
Proposition~8.10, and Appendix~B]{Murf} or~\cite[Section~B.2.5]{Psing}).
 The aim of this section is to extend this construction from Noetherian
to coherent commutative rings, and further to the relative situation
with a flat commutative algebra $R$ over a coherent commutative
ring $A$ with a dualizing complex~$D^\bu$.
 The tensor category structure with the unit object $T^\bu=R\ot_A D^\bu$
will be defined on the $R/A$\+semicoderived category
$\sD^\sico_A(R\modl)$ in the relative situation.

 The construction is based on the results of
Sections~\ref{covariant-duality}\+-\ref{semiderived-categories}.
 As in these sections, we will have to make the technical assumption
that fp\+injective $A$\+modules have finite injective dimensions.
 Throughout this section, the term \emph{dualizing complex of
$A$\+modules} is understood in the sense of
Example~\ref{flat-tensor-example}(2); i.~e., it is a complex of
$A$\+modules which, viewed as a complex of $A$\+$A$\+bimodules,
is a dualizing complex for the rings $A$ and $A$, in the sense
of the definition in Section~\ref{covariant-duality}.

\begin{thm}
 Let $A$ be a coherent commutative ring such that fp\+injective
$A$\+modules have finite injective dimensions, and let $D^\bu$ be
a dualizing complex of $A$\+modules.
 Then there is a natural associative, commutative, and unital tensor
category structure on the coderived category of $A$\+modules, provided
by the operation of \emph{(derived) cotensor product} of
complexes of $A$\+modules
$$
 \oc_{D^\bu}\:\sD^\co(A\modl)\times\sD^\co(A\modl)
 \lrarrow\sD^\co(A\modl).
$$
 The dualizing complex $D^\bu\in\sD^\b_\fp(A\modl)\subset
\sD^\co(A\modl)$ is the unit object of this tensor category structure.
\end{thm}

\begin{proof}
 We use the results of Section~\ref{covariant-duality} in order to identify
the coderived category of $A$\+modules $\sD^\co(A\modl)$ with
the absolute derived category of flat $A$\+modules
$\sD^\abs(A\modl_\fl)$.
 Indeed, by Theorem~\ref{derived-co-contra}, the choice of a dualizing
complex $D^\bu$ induces an equivalence between the coderived category
and the contraderived category of $A$\+modules, $\sD^\co(A\modl)\simeq
\sD^\ctr(A\modl)$, and by 
Theorem~\ref{contraderived-flat-projectives}(a),
the contraderived category of $A$\+modules is equivalent to
the contraderived category of flat $A$\+modules,
$\sD^\ctr(A\modl)\simeq\sD^\ctr(A\modl_\fl)$.
 By Proposition~\ref{cfh-generalized}, flat $A$\+modules have finite
projective dimensions, and by~\cite[Remark~2.1]{Psemi}, it follows
that contraacyclic complexes of flat $A$\+modules are absolutely
acyclic with respect to the exact category of flat $A$\+modules,
so the contraderived and the absolute derived categories of flat
$A$\+modules coincide, $\sD^\ctr(A\modl_\fl)=\sD^\abs(A\modl_\fl)$
(see~\cite[Section~A.1]{Pcosh} for the definition of the absolute
derived category).

 Notice that the tensor product of a complex of flat $A$\+modules
and an absolutely acyclic complex of flat $A$\+modules is obviously
an absolutely acyclic complex of flat $A$\+modules, so there is
the tensor product functor
$$
 \ot_A\:\sD^\abs(A\modl_\fl)\times\sD^\abs(A\modl_\fl)\lrarrow
 \sD^\abs(A\modl_\fl).
$$
 By the definition, the cotensor product functor $\oc_{D^\bu}$ is
obtained from this functor of tensor product of complexes of
flat $A$\+modules using the equivalence of categories
$\sD^\co(A\modl)\simeq\sD^\abs(A\modl_\fl)$.
 Explicitly, this means that in order to compute the cotensor product
of two objects of the coderived category, one has to represent them
by two complexes of injective $A$\+modules $I^\bu$ and $J^\bu$
(using the result of Theorem~\ref{coderived-homotopy-injectives})
and then apply the formula
$$
 I^\bu\oc_{D^\bu}J^\bu = D^\bu\ot_A\Hom_A(D^\bu,I^\bu)\ot_A
 \Hom_A(D^\bu,J^\bu).
$$

 Furthermore, there is a well-defined functor of tensor product of
complexes of flat $A$\+modules and arbitrary complexes of $A$\+modules
$$
 \ot_A\:\sD^\abs(A\modl_\fl)\times\sD^\co(A\modl)\lrarrow
 \sD^\co(A\modl),
$$
since the tensor product of an absolutely acyclic complex of flat
$A$\+modules with any complex of $A$\+modules is a coacyclic (in fact,
absolutely acyclic) complex of $A$\+modules, as is the tensor product
of a complex of flat $A$\+modules with a coacyclic complex of
$A$\+modules.
 The cotensor product of two complexes of injective $A$\+modules
can be alternatively defined by the rules
$$
 I^\bu\oc_{D^\bu}J^\bu\simeq\Hom_A(D^\bu,I^\bu)\ot_AJ^\bu
 \simeq I^\bu\ot_A\Hom_A(D^\bu,J^\bu),
$$
where the isomorphism signs denote natural isomorphisms in
the absolute derived category of fp\+injective $A$\+modules
(cf.~\cite[Proposition~8.10]{Murf} and~\cite[Section~B.2.5]{Psing}).
\end{proof}

 The following theorem is the main result of this section.

\begin{thm}
 Let $A$ be a coherent commutative ring such that fp\+injective
$A$\+modules have finite injective dimensions, and let $D^\bu$ be
a dualizing complex of $A$\+modules.
 Let $A\rarrow R$ be a morphism of commutative rings making
$R$ a flat $A$\+module.
 Then there is a natural associative, commutative, and unital tensor
category structure on the $R/A$\+semicoderived category of
$R$\+modules, provided by the functor of \emph{(derived) semitensor
product} of complexes of $R$\+modules
$$
 \os_{T^\bu}\:\sD^\sico_A(R\modl)\times\sD^\sico_A(R\modl)
 \lrarrow\sD^\sico_A(R\modl),
$$
where $T^\bu=R\ot_AD^\bu$ is the relative dualizing complex
for the ring homomorphism $A\rarrow R$ associated with the dualizing
complex $D^\bu$ for the ring $A$.
 The relative dualizing complex $T^\bu\in\sD^\sico_A(R\modl)$ is
the unit object of this tensor category structure.
\end{thm}

\begin{proof}
 To constuct the functor $\os_{T^\bu}$, we will use the results of
Section~\ref{semiderived-categories} in order to identify
the semicoderived category $\sD^\sico_A(R\modl)$ with the quotient
category of homotopy category of complexes of $A$\+flat $R$\+modules
by its thick subcategory of complexes that are \emph{absolutely acyclic
as complexes of flat $A$\+modules}.
 Let us call the latter category simply the \emph{semiderived} or
the \emph{$R/A$\+semiderived category of $A$\+flat $R$\+modules} and
denote it by $\sD^\si_A(R\modl_{A\dfl})$.

 Indeed, by Theorem~\ref{relative-co-contra}, the choice of
the relative dualizing complex $D^\bu\rarrow T^\bu$ for a ring
homomorphism $A\rarrow R$ induces an equivalence between
the $R/A$\+semicoderived category and the $R/A$\+semicontraderived
category of $R$\+modules,
$\sD^\sico_A(R\modl)\simeq\sD^\sictr_A(R\modl)$.
 By Theorem~\ref{semiderived-fp-injective-flat}(b),
the $R/A$\+semicontraderived category of $R$\+modules is
equivalent to the $R/A$\+semicontraderived category of $A$\+flat
$R$\+modules, $\sD^\sictr_A(R\modl)\simeq\sD^\sictr_A(R\modl_{A\dfl})$,
and since the exact category of flat $A$\+modules has finite
homological dimension, the $R/A$\+semicontraderived category
of $A$\+flat $R$\+modules coincides with their $R/A$\+semiderived
category, $\sD^\sictr_A(R\modl_{A\dfl})=\sD^\si_A(R\modl_{A\dfl})$.
{\hbadness=1400\par}

 Composing these equivalences, we obtain an equivalence of
triangulated categories $\sD^\si_A(R\modl_{A\dfl})\simeq
\sD^\sico_A(R\modl)$ provided by the functor assigning to
a complex of $A$\+flat $R$\+modules $G^\bu$ the complex of
$R$\+modules $T^\bu\ot_RG^\bu\simeq D^\bu\ot_AG^\bu$.
 The inverse functor assigns to a complex of $A$\+injective
$R$\+modules $J^\bu$ the complex of $A$\+flat $R$\+modules
$\Hom_R(T^\bu,J^\bu)\simeq\Hom_A(D^\bu,J^\bu)$.

 Furthermore, it will be more convenient for us to work with
the \emph{$R/A$\+semicoderived category of $A$\+flat
$R$\+modules} $\sD^\sico(R\modl_{A\dfl})$, defined as the quotient
category of the homotopy category of $A$\+flat $R$\+modules
by its thick subcategory of complexes that are \emph{coacyclic
as complexes of flat $A$\+modules}.
 Since, in view of Proposition~\ref{cfh-generalized}, the exact
category of flat $A$\+modules has finite projective dimension,
by~\cite[Remark~2.1]{Psemi} the $R/A$\+semiderived category of $A$\+flat
$R$\+modules coincides with their $R/A$\+semicoderived category,
$\sD^\si_A(R\modl_{A\dfl})=\sD^\sico_A(R\modl_{A\dfl})$.

 Let us emphasize that while, by
Theorem~\ref{contraderived-flat-projectives}(a), any complex of
flat $A$\+modules that is contraacyclic as a complex of $A$\+modules
is also contraacyclic as a complex of flat $A$\+modules, there is
\emph{no} claim that a complex of flat $A$\+modules coacyclic as
a complex of $A$\+modules should be coacyclic as a complex of
flat $A$\+modules (cf.~\cite[Chapter~3]{Murf}
and~\cite[Remark~1.5 and Section~2.5]{Psing}).
 Whenever below we mention a ``coacyclic complex of flat $A$\+modules'',
it means a complex coacyclic with respect to the exact category of
flat $A$\+modules.

\begin{prop} \label{left-derived-tensor-products}
 Let $A\rarrow R$ be an arbitrary morphism of commutative rings
such that $R$ is a flat $A$\+module.
 Then one can construct the left derived functor of tensor product of
complexes of $A$\+flat $R$\+modules
$$
 \ot_R^\boL\:\sD^\sico_A(R\modl_{A\dfl})\times
 \sD^\sico_A(R\modl_{A\dfl})\lrarrow\sD^\sico_A(R\modl_{A\dfl})
$$
providing an associative, commutative, and unital tensor category
structure on the $R/A$\+semicoderived category of $A$\+flat
$R$\+modules\/ $\sD^\sico_A(R\modl_{A\dfl})$
with the unit object $R\in\sD^\sico_A(R\modl_{A\dfl})$.

 In addition, one can construct the left derived functor
of tensor product of complexes of $A$\+flat $R$\+modules and
arbitrary complexes of $R$\+modules
$$
 \ot_R^\boL\:\sD^\sico_A(R\modl_{A\dfl})\times
 \sD^\sico_A(R\modl)\lrarrow\sD^\sico_A(R\modl)
$$
providing an associative and unital module category structure over
the tensor category\/ $\sD^\sico_A(R\modl_{A\dfl})$ on
the $R/A$\+semicoderived category of $R$\+modules\/
$\sD^\sico_A(R\modl)$.
\end{prop}

 When $A$ is a coherent ring such that fp\+injective $A$\+modules
have finite injective dimensions and a dualizing complex $D^\bu$ is
chosen for the ring $A$, the equivalence of categories
$\sD^\sico_A(R\modl_{A\dfl})\simeq\sD^\sico_A(R\modl)$ will transform
both of the left derived functors of
Proposition~\ref{left-derived-tensor-products}
into the desired functor of semitensor product~$\os_{T^\bu}$.

\begin{proof}[Proof of Proposition~\ref{left-derived-tensor-products}]
 In order to obtain the derived functors $\ot_R^\boL$, we will
apply the general construction of balanced derived functors of
functors of two arguments elaborated in~\cite[Lemma~2.7]{Psemi}.
 Let us call a complex of $A$\+flat $R$\+modules $F^\bu$
\emph{relatively homotopy $R$\+flat} if for any complex of $A$\+flat
$R$\+modules $L^\bu$ that is coacyclic as a complex of flat $A$\+modules
the complex of $R$\+modules $F^\bu\ot_RL^\bu$ is coacyclic as a complex
of flat $A$\+modules, and for any complex of $R$\+modules $N^\bu$ that
is coacyclic as a complex of $A$\+modules the complex of $R$\+modules
$F^\bu\ot_RN^\bu$ is coacyclic as a complex of $A$\+modules.

 Similarly, let us call a complex of $R$\+modules $H^\bu$
\emph{homotopy $R/A$\+flat} if for any complex of $A$\+flat
$R$\+modules $M^\bu$ that is coacyclic as a complex of flat
$A$\+modules the complex of $R$\+modules $M^\bu\ot_RH^\bu$
is coacyclic as a complex of $A$\+modules.
 Both the relatively homotopy $R$\+flat complexes of $A$\+flat
$R$\+modules and the homotopy $R/A$\+flat complexes of $R$\+modules
are thought of as ``homotopy flat in the direction of $R$ relative
to~$A$'' (while the former are supposed to be also complexes of
flat $A$\+modules, and the latter ones are not).

\begin{lem} \label{enough-relatively-homotopy-flats}
 Let $A\rarrow R$ be a morphism of commutative rings making $R$
a flat $A$\+module.  Then \par
\textup{(a)} for any complex of $A$\+flat $R$\+modules $M^\bu$
there exists a relatively homotopy $R$\+flat complex of $A$\+flat
$R$\+modules $F^\bu$ together with a morphism of complexes of
$R$\+modules $F^\bu\rarrow M^\bu$ whose cone is coacyclic
as a complex of flat $A$\+modules; \par
\textup{(b)} for any complex of $R$\+modules $N^\bu$ there exists
a homotopy $R/A$\+flat complex of $R$\+modules $H^\bu$ together
with a morphism of complexes of $R$\+modules $H^\bu\rarrow N^\bu$
whose cone is coacyclic as a complex of $A$\+modules.
\end{lem}

\begin{proof}
 Part~(a): notice that the full subcategory of relatively homotopy
$R$\+flat complexes of $A$\+flat $R$\+modules is closed under
the operations of shift, cone, and the passage to an infinite direct
sum (and consequently, also to the countable homotopy direct limit)
in the homotopy category of complexes of $A$\+flat $R$\+modules
$\Hot(R\modl_{A\dfl})$.
 Besides, any complex of $R$\+modules $R\ot_AE^\bu$ induced from
a complex of flat $A$\+modules $E^\bu$ is a relatively homotopy
$R$\+flat complex of $A$\+flat $R$\+modules.
 Indeed, the tensor product of a complex of flat $A$\+modules with
a coacyclic complex of flat $A$\+modules is a coacyclic complex of
flat $A$\+modules, and the tensor product of a complex of flat
$A$\+modules with a coacyclic complex of $A$\+modules is a coacyclic
complex of $A$\+modules.

 Now, given a complex of $A$\+flat $R$\+modules $M^\bu$, consider
the bar-complex
$$
 \dotsb\lrarrow R\ot_AR\ot_AR\ot_AM^\bu\lrarrow R\ot_AR\ot_AM^\bu
 \lrarrow R\ot_AM^\bu.
$$
 Let $F^\bu$ be the total complex of this bicomplex obtained by
taking infinite direct sums along the diagonals.
 On the one hand, the complex $F^\bu$ is homotopy equivalent to
the homotopy direct limit of the total complexes of the finite
segments (subcomplexes of silly filtration) of the bar-bicomplex,
which are obtained from complexes of $R$\+modules induced from
complexes of flat $A$\+modules by finite iterations of
the operations of shift and cone.
 So the complex $F^\bu$ is a relatively homotopy $R$\+flat complex
of $A$\+flat $R$\+modules.
 On the other hand, the cone of the natural morphism $F^\bu
\rarrow M^\bu$ is not only coacyclic, but even contractible
as a complex of (flat) $A$\+modules.

 Part~(b): the full subcategory of homotopy $R/A$\+flat complexes
of $R$\+modules is closed under the operations of shift, cone, and
the passage to an infinite direct sum in the homotopy category of
complexes of $R$\+modules $\Hot(R\modl)$.
 Besides, any complex of $R$\+modules $R\ot_A C^\bu$ induced from
a complex of $A$\+modules $C^\bu$ is homotopy $R/A$\+flat.
 Indeed, the tensor product of a coacyclic complex of flat
$A$\+modules with any complex of $A$\+modules is a coacyclic
complex of $A$\+modules.

 Now, given a complex of $R$\+modules $N^\bu$, consider
the bar-complex
$$
 \dotsb\lrarrow R\ot_AR\ot_AR\ot_AN^\bu\lrarrow R\ot_AR\ot_AN^\bu
 \lrarrow R\ot_AN^\bu.
$$
 Let $H^\bu$ be the total complex of this bicomplex obtained by
taking infinite direct sums along the diagonals.
 On the one hand, the complex $H^\bu$ is homotopy equivalent to
the homotopy direct limit of the subcomplexes of silly filtration
on the bar-bicomplex, which are obtained from complexes of
$R$\+modules induced from complexes of $A$\+modules by finite
iterations of the operations of shift and cone.
 So $H^\bu$ is a homotopy $R/A$\+flat complex of $R$\+modules.
 On the other hand, the cone of the natural morphism $H^\bu\rarrow
N^\bu$ is a contractible complex of $A$\+modules.
 (Cf.~\cite[Theorem~2.6]{Psemi}.)
\end{proof}

 Now we can finish the proof of Proposition~\ref{left-derived-tensor-products}.
 Given two complexes of $A$\+flat $R$\+modules $M^\bu$ and
$N^\bu$, one replaces one or both of them with relatively homotopy
$R$\+flat complexes of $A$\+flat $R$\+modules $F^\bu$ and/or $G^\bu$
endowed with morphisms of complexes of $R$\+modules $F^\bu\rarrow M^\bu$
and $G^\bu\rarrow N^\bu$ with the cones coacyclic as complexes of flat
$A$\+modules.
 The induced morphisms of complexes of $A$\+flat $R$\+modules
$F^\bu\ot_R G^\bu\rarrow F^\bu\ot_R N^\bu$ and $F^\bu\ot_R G^\bu\rarrow
M^\bu\ot_R G^\bu$ have cones coacyclic with respect to the exact
category of flat $A$\+modules, so either of the three complexes
$F^\bu\ot_R N^\bu$, \ $M^\bu\ot_R G^\bu$, or $F^\bu\ot_R G^\bu$ can be
taken as representing the object $M^\bu\ot_R^\boL N^\bu$ in
the semicoderived category $\sD^\sico_A(R\modl_{A\dfl})$.

 Similarly, given a complex of $A$\+flat $R$\+modules $M^\bu$ and
an arbitrary complex of $R$\+modules $N^\bu$, one either replaces
$M^\bu$ with a relatively homotopy $R$\+flat complex of $A$\+flat
$R$\+modules $F^\bu$ endowed with a morphism of complexes of
$R$\+modules $F^\bu\rarrow M^\bu$ with the cone coacyclic as
a complex of flat $A$\+modules, or replaces $N^\bu$ with a homotopy
$R/A$\+flat complex of $R$\+modules $H^\bu$ endowed with a morphism
of complexes of $R$\+modules $H^\bu\rarrow N^\bu$ with the cone
coacyclic as a complex of $A$\+modules.
 The induced morphisms of complexes of $R$\+modules 
$F^\bu\ot_R H^\bu\rarrow F^\bu\ot_R N^\bu$ and $F^\bu\ot_R H^\bu\rarrow
M^\bu\ot_R H^\bu$ have cones coacyclic over~$A$, so any one of
the three complexes $F^\bu\ot_R N^\bu$, \ $M^\bu\ot_R H^\bu$, or
$F^\bu\ot_R H^\bu$ can be taken as representing the object
$M^\bu\ot_R^\boL N^\bu$ in the semicoderived category
$\sD^\sico_A(R\modl)$.

 The left derived functors $\ot_R^\boL$ are well-defined by this
construction according to~\cite[Lemma~2.7]{Psemi}, whose conditions
are satisfied by Lemma~\ref{enough-relatively-homotopy-flats}
and~\cite[Lemma~1.6]{Pkoszul}.
 To sum up this somewhat tedious construction, one can simply say
that both the derived functors $M^\bu\ot_R^\boL N^\bu$ are computed
by the total complex of the bar-tricomplex
$$
 \dotsb\lrarrow M^\bu\ot_AR\ot_AR\ot_AN^\bu\lrarrow M^\bu\ot_AR\ot_AN^\bu
 \lrarrow M^\bu\ot_AN^\bu
$$
constructed by taking infinite direct sums along the diagonal
hyperplanes.
\end{proof}

 Finally, we can return to the situation with a coherent commutative ring
$A$ such that fp\+injective $A$\+modules have finite injective
dimensions and a dualizing complex of $A$\+modules $D^\bu$ is chosen.
 As above, let $R$ be a flat commutative $A$\+algebra and
$T^\bu=R\ot_AD^\bu$ be the corresponding relative dualizing complex.
 Given two complexes of $R$\+modules $M^\bu$ and $N^\bu$, one replaces
one or both of them with complexes of $A$\+injective $R$\+modules
$I^\bu$ and/or $J^\bu$ endowed with morphisms of complexes of
$R$\+modules $M^\bu\rarrow I^\bu$ and $N^\bu\rarrow J^\bu$ with
the cones coacyclic as complexes of $A$\+modules.
 Then the object of the semicoderived category $M^\bu\os_{T^\bu}N^\bu$
can be computed in either of three equivalent ways as
\begin{gather*}
 D^\bu\ot_A(\Hom_A(D^\bu,I^\bu)\ot_R^\boL\Hom_A(D^\bu,J^\bu)) \\
 \simeq \Hom_A(D^\bu,I^\bu)\ot_R^\boL N^\bu\simeq
 \Hom_A(D^\bu,J^\bu)\ot_R^\boL M^\bu,
\end{gather*}
where the isomorphism signs mean natural isomorphisms in
the semicoderived category $\sD^\sico_A(R\modl)$.
 As the equivalence of categories $\sD^\sico_A(R\modl_{A\dfl})\simeq
\sD^\sico_A(R\modl)$ takes the one-term complex
$R\in\sD^\sico_A(R\modl_{A\dfl})$ to the relative dualizing complex
$T^\bu\in\sD^\sico_A(R\modl)$, the relative dualizing complex $T^\bu$
is the unit object of the tensor category structure $\os_{T^\bu}$
on the triangulated category $\sD^\sico_A(R\modl)$.
\end{proof}

\end{document}